\documentclass[a4paper,10pt]{article}
\usepackage[T1]{fontenc}
\usepackage[utf8]{inputenc}
\usepackage[english,greek,frenchb,french]{babel}
\usepackage{url,csquotes}
\usepackage[hidelinks,hyperfootnotes=false]{hyperref}
\usepackage{float}
\usepackage{amsfonts,amssymb,enumerate}
\usepackage{amsthm}
\usepackage{amsmath}
\usepackage{mathtools,thmtools}
\allowdisplaybreaks[1]
\usepackage{graphicx}
\usepackage{bbm}
\usepackage{listings}
\usepackage{titling}
\usepackage{dsfont}
\usepackage{color}

\newcommand{\E}{\mathbb{E}}
\newcommand{\cE}{\mathcal{E}}
    \newcommand{\Prb}{\mathbb{P}}
	
		\newcommand{\cG}{\mathcal{G}}
				\newcommand{\cA}{\mathcal{A}}
			\newcommand{\cM}{\mathcal{M}}
			\newcommand{\fG}{\mathfrak{G}}
\newcommand{\cV}{\mathcal{V}_2}
\newcommand{\cL}{\mathcal{L}}
\newcommand{\cI}{\mathcal{I}}
\newcommand{\cW}{\mathcal{W}}
\newcommand{\vv}{\overrightarrow{v}}
\newcommand{\uu}{\overrightarrow{u}}
\newcommand{\ff}{\smash{f_n}}
		\newcommand{\cH}{\mathcal{H}}
		\newcommand{\cP}{\mathcal{P}}
	
	\newcommand{\sR}{\mathbb{R}}
	\newcommand{\sN}{\mathbb{N}}
		\newcommand{\sS}{\mathbb{S}}

			\DeclareMathOperator{\disc}{disc}

				\DeclareMathOperator{\card}{card}
					\DeclareMathOperator{\mincut}{Mincut}
		\DeclareMathOperator{\dive}{div}
		\DeclareMathOperator{\flow}{flow}
		\DeclareMathOperator{\cyl}{cyl}
		
    \newcommand{\sZ}{\mathbb{Z}}
    \newcommand{\sC}{\mathcal{C}}
    \newcommand{\ep}{\varepsilon}    

\declaretheorem[name=Theorem,within=section]{thm}
\declaretheorem[name=Proposition,numberlike=thm]{prop}
\declaretheorem[name=Definition,numberlike=thm]{defn}
\declaretheorem[name=Lemma,numberlike=thm]{lem}

\author{Barbara Dembin \thanks{LPSM UMR 8001, Universit{\'e} Paris Diderot, Sorbonne Paris Cit{\'e}, CNRS, F-75013 Paris, France}}

\newlength{\separationtitre}

\date{}
\begin{document}
\newpage

 \selectlanguage{english}

\title{The maximal flow from a compact convex subset to infinity in first passage percolation on $\sZ^d$ 
}
\maketitle
\paragraph{Abstract:} We consider the standard first passage percolation model on $\sZ^d$ with a distribution $G$ on $\sR^+$ that admits an exponential moment. We study the maximal flow between a compact convex subset $A$ of $\sR ^d$ and infinity. The study of maximal flow is associated with the study of sets of edges of minimal capacity that cut $A$ from infinity. We prove that the rescaled maximal flow between $nA$ and infinity $\phi(nA)/n^{d-1}$ almost surely converges  towards a deterministic constant depending on $A$. This constant corresponds to the capacity of the boundary $\partial A$ of $A$ and is the integral of a deterministic function over $\partial A$. This result was shown in dimension $2$ and conjectured for higher dimensions by Garet in \cite{Garet2}.
\newline



\section{Introduction}
The model of first passage percolation was first introduced by Hammersley and Welsh \cite{HammersleyWelsh} in 1965 as a model for the spread of a fluid in a porous medium. In this model, mathematicians studied intensively geodesics, \textit{i.e.},  fastest paths between two points in the grid. The study of maximal flows in first passage percolation started later in 1984 in dimension $2$ with an article of Grimmett and Kesten \cite{GrimmettKesten84}. In 1987, Kesten studied maximal flows in dimension $3$ in \cite{Kesten:flows}. The study of maximal flows is associated with the study of random cutsets that can be seen as $(d-1)$-dimensional surfaces. Their study presents more technical difficulties than the study of geodesics. Thus, the interpretation of first passage percolation in terms of maximal flows has been less studied. 

To each edge in the graph $\sZ^d$, we assign a random i.i.d. capacity with distribution $G$ on $\sR^+$ that admits an exponential moment. We interpret this capacity as a rate of flow, \textit{i.e.}, it corresponds to the maximal amount of water that can cross the edge per second. Let us consider a compact convex subset $A$ of $\sR^d$. We interpret the set $A$ as a source of water. We are interested in the maximal amount of water that can flow from the boundary $\partial A$ of $A$ to infinity per second. This issue is in fact analogous to the study of the smallest capacity $\mincut(A,\infty)$ over sets of edges separating $A$ from infinity. This issue was first studied in dimension $2$ by Garet in \cite{Garet2}, he proved that the rescaled maximal flow between $nA$ and infinity $\phi(nA)/n$ almost surely converges towards an integral of a deterministic function $\nu$ over $\partial A$.

Several issues arise when we study this problem in higher dimensions. Garet proves his result in \cite{Garet2} by proving separately upper and lower large deviations. Although the proof of the upper large deviations may be adapted to higher dimensions, the proof of lower large deviations strongly relies on combinatorial estimates that fail in higher dimensions. Moreover, in dimension $2$, the function $\nu$ is actually simpler to study. Thanks to the duality, it is related to the study of geodesics, whereas in higher dimensions, we cannot avoid the study of random surfaces to define the function $\nu$. To get a better understanding of this deterministic function in higher dimensions, we first study the maximal flow in a box. Let us consider a large box in $\sZ^d$ oriented along a given direction $v$. Next, we consider the two opposite sides of the box normal to $v$ that we call top and bottom. We are interested in the maximal flow that can cross the box from its top to its bottom per second.  More precisely, we can ask if this maximal flow properly renormalized converges when the size of the box grows to infinity. This question was addressed in \cite{Kesten:flows},  \cite{Rossignol2010} and \cite{Zhang2017} where one can find laws of large numbers and large deviation estimates for this maximal flow when the dimensions of the box grow to infinity under some moments assumptions on the capacities and on the direction $v$. The maximal flow properly renormalized converges towards the so-called flow constant $\nu(v)$.  In \cite{flowconstant}, Rossignol and Th{\'e}ret proved the same results without any moment assumption on $G$ for any direction $v$. Roughly speaking, the flow constant $\nu(v)$ corresponds to the expected maximal amount of water that can flow per second in the direction of $v$. 
Let us consider a point $x$ in $\partial A$ with its associated normal unit exterior vector $n_A(x)$ and infinitesimal surface $S(x)$ around $x$. When we consider  $nA$, an enlarged version of $A$, the surface $S(x)$ becomes $nS(x)$ and the expected maximal amount of water that can flow in the box of basis $nS(x)$ in the direction $n_A(x)$ is of order $n^{d-1}\nu(n_A(x))C_S$ where $C_S$ is a constant depending on the area of the surface. Heuristically, when we sum over points in $\partial A$, we obtain that the maximal flow  between $nA$ and infinity $\phi(nA)$ is roughly $n^{d-1}$ times the integral of $\nu$ over $\partial A$.  

The aim of this paper is to prove the following theorem that was conjectured by Garet in \cite{Garet2}.
\begin{thm}\label{thm1} Let $d\geq 3$. Let $A$ be a compact convex subset of $\sR^d$. Let $G$ be a probability measure on $[0,+\infty[$ such that $G(\{0\})< 1-p_c(d)$. Let $\nu$ be the flow constant associated to $G$ and let
$$\phi_A=\int_{\partial A}\nu(n_A(x))d\cH^{d-1}(x)\,.$$
For each $\ep>0$, there exist positive constants $C_1$ and $C_2$ depending only on $\ep$ and $G$, such that for all $n\geq 0$,
$$\Prb\left(\,\left| \frac{\mincut(nA,\infty)}{n^{d-1}}-\phi_A\right|\geq \ep\,\right)\leq C_1\exp(-C_2n^{d-1})\,.$$
\end{thm}

As a corollary, $\mincut(nA,\infty)/n^{d-1}$ converges in probability towards $\phi_A$ when $n$ goes to infinity. Roughly speaking, the rescaled maximal flow that can go from $nA$ to infinity is limited by the capacity of $\partial A$, or equivalently, the rescaled minimal capacity of a cutset between $nA$ and infinity is equal to the capacity of $\partial A$. In addition, we shall prove that there exists a minimal cutset $E$ between the set $nA$ and infinity, \textit{i.e.}, such that the capacity of $E$ is equal to $\phi(nA)$ and $E$ separates $nA$ from infinity. This is far from obvious, but it is a natural consequence of Zhang's result \cite{Zhang2017}. Indeed, for a fixed $n$, there may exist a sequence of sets $(\cE_p)_{p\in \sN}$ of growing size such that $\cE_p\subset\E ^d$ cuts $A$ from infinity and $$\lim_{p\rightarrow \infty} V(\cE_p)=\mincut(nA,\infty)\,.$$ There is no direct argument that allows to extract a sequence from $(\cE_p)_{p\in \sN}$ which would converge to a cutset realizing the minimum. When we consider cutsets in a bounded region, the existence of a cutset achieving the infimum becomes trivial as the number of possible cutsets is finite. 
 We define the edge boundary of $A$ as $$\partial_e A=\big\{\,e=\langle x,y\rangle\in \E^d\,:\, x\in A\cap \sZ^d, \,y\in \sZ^d\setminus A\,\big\}\,.$$
 
\begin{thm}[Existence of a minimal cutset and control of its size]\label{thmZhang} 
Let $A$ be a compact convex subset of $\sR^d$ containing the origin. Let $G$ be a probability measure on $[0,+\infty[$ such that $G(\{0\})< 1-p_c(d)$ and $G$ admits an exponential moment.
\begin{enumerate}
\item With probability 1, there exists a minimal cutset from $A$ to infinity in the original lattice $(\sZ^d,\E^d)$. 
\item There exist constants $\beta_0$, $C_1$, $C_2$ and $\lambda$ depending only on $d$ and $G$ such that for any $\beta>\beta_0$, for any $n\geq \lambda |\partial_e A|$, 
$$\Prb[\text{All the minimal cutsets $E$ are such that $|E|\geq\beta n$}]\leq C_1 \exp(-C_2\beta n)\,.$$
\end{enumerate}
\end{thm}
We prove Theorem \ref{thm1} by proving separately the upper large deviations above the constant $\cI(A)$ in Theorem \ref{ULD} and the lower large deviations below the constant $\phi_A$ in Theorem \ref{LLD}. It will be more convenient in the following to work in the graph $(\sZ^d_n,\E^d_n)$ having for vertices $\sZ^d _n=\sZ^d/n$ and for edges $\E_n^d$, the set of pairs of points of $\sZ^d$ at distance $1/n$ from each other. In this setting, the set $A$ remains fixed and the lattice shrinks. We denote by $\mincut_n(A,\infty)$, the minimal capacity over sets of edges in $\E_n ^d$ separating $A$ from infinity. We define 
$$\cI(A)=\int_{\partial A}\nu(n_A(x))d\cH^{d-1}(x)\,.$$
The quantity $\cI(A)$ may be interpreted as the capacity of $\partial A$.
\begin{thm}[Upper large deviations]\label{ULD}
Let $A$ be a compact convex subset of $\sR^d$ containing the origin. Let $G$ be a probability measure on $[0,+\infty[$ such that $G(\{0\})< 1-p_c(d)$ and $G$ admits an exponential moment. For each $\lambda>\cI(A)$, there exist positive constants $C_1$ and $C_2$ depending only on $\lambda$, $A$ and $G$, such that for all $n\geq 0$,
$$\Prb\left[\mincut_n(A,\infty)\geq \lambda n^{d-1} \right]\leq C_1\exp(-C_2n^{d-1})\,.$$
\end{thm}
The proof of Theorem \ref{ULD} is inspired by the proof of the enhanced large deviations upper bound in \cite{Garet2} and the proof of the upper large deviations for the maximal flow through a domain of $\sR^d$ done in \cite{CerfTheret09geoc}. Roughly speaking, the idea is to build a cutset $E$ from $A$ to infinity whose capacity is close to $\cI(A)n  ^{d-1}$ and next to bound the probability that $\mincut_n(A,\infty)$ is abnormally big, \textit{i.e.}, greater than $\cI(A)n^{d-1}$, by the probability that the capacity of $E$ is abnormally big. To do so, we first approximate $A$ from the outside by a convex polytope $P$. For each face $F$ of $P$ and $v$ its associated exterior unit normal vector, we consider the cylinder $\cyl(F+\ep v,\ep)$ of basis $F+\ep v$ and of height $\ep>0$ and a cutset from the top to the bottom of the cylinder having minimal capacity. We build $E$ by merging the cutsets associated to all the faces of the polytope.  The union of these cutsets is not yet a cutset itself because of the potential holes between these cutsets. We fix this issue by adding extra edges to fill the holes. We next control the number of extra edges we have added. We also need to control the capacity of the cutsets in a cylinder of polyhedral basis to obtain the desired control.

\begin{thm}[Lower large deviations]\label{LLD}
Let $A$ be a compact convex subset of $\sR^d$. Let $G$ be a probability measure on $[0,+\infty[$ such that $G(\{0\})< 1-p_c(d)$ and $G$ admits an exponential moment. We define
$$\phi_A=\inf\big\{\, \cI(S)\,:\, A\subset S \text{ and $S$ is compact}\,\big\}\,.$$ For each $\lambda<\phi_A$, there exist positive constants $C_1$ and $C_2$ depending only on $\lambda$, $A$, and $G$, such that for all $n\geq 0$,
$$\Prb\left[ \mincut_n(A,\infty)\leq \lambda n^{d-1} \right]\leq C_1\exp(-C_2n^{d-1})\,.$$
\end{thm}
To study the lower large deviations, we try to proceed as in the work of Cerf and Th{\'e}ret in \cite{CerfTheret09infc}. The idea is to create from a cutset $E\subset \E_n^d$ that cuts $A$ from infinity a continuous subset of $\sR^d$ whose edge boundary (\textit{i.e.}, the edges that have one extremity in the continuous subset and the other one outside) corresponds to the cutset $E$. As we can control the number of edges in a minimal cutset thanks to the work of Zhang \cite{Zhang2017}, we can consider a cutset from $A$ to infinity of minimal capacity and that has at most $cn ^{d-1}$ edges with high probability, for some positive constant $c$. Thanks to this crucial result, the continuous set we build has a perimeter at most $c$. In \cite{CerfTheret09infc}, as the two authors work in a compact region $\Omega$, the continuous object they obtain live in the compact space consisting of all subsets of $\Omega$ of perimeter less than or equal to $c$. In our context, as our cutset $E$ can go potentially very far from $A$, we cannot build from $E$ a continuous set that belongs to some compact space and therefore we cannot use the same method as in \cite{CerfTheret09infc}. However, as the capacity of $E$ is small, we expect it to remain close to the boundary of $\partial A$. We should observe unlikely events just by inspecting what happens near the boundary of $A$. This will enable us to study only the portion of the cutset $E$ near $\partial A$ and to define a continuous version of this portion that belongs to a compact set. Starting from there, we can follow the strategy of \cite{CerfTheret09infc}.

 Finally, we prove in Proposition \ref{prp} that the two constants $\cI(A)$ and $\phi_A$ appearing in Theorems \ref{ULD} and \ref{LLD} are equal. This yields the result stated in Theorem \ref{thm1}.
\begin{prop}\label{prp}
Let $A$ be a compact convex subset of $\sR^d$. The minimal capacity $\phi_A$ for the flow from $A$ to infinity is achieved by $\cI(A)$, the capacity of the boundary of $A$, \textit{i.e.}, 
$$\phi_A=\inf\big\{\, \cI(S),\, A\subset S \text{ and $S$ is compact}\,\big\}=\cI(A)\,.$$
\end{prop}

The rest of the paper is organized as follows. In section \ref{s2}, we present the model. In section \ref{s3} and \ref{s4}, we give all the necessary definitions and background. In section \ref{sectionULD}, we prove the upper large deviations Theorem \ref{ULD}. We prove the existence of a minimal cutset Theorem \ref{thmZhang} in section \ref{s6} and  the lower large deviations Theorem \ref{LLD} in section \ref{sectionLLD}. Finally, we conclude the proof of Theorem \ref{thm1} by proving Proposition \ref{prp} in section \ref{sectionconclusion}. 

\section{The model}\label{s2}
\subsection{The environment}
 Let $n\geq 1$ be an integer. We consider the graph $(\sZ^d_n,\E^d_n)$ having for vertices $\sZ^d _n=\sZ^d/n$ and for edges $\E_n^d$, the set of pairs of points of $\sZ^d$ at distance $1/n$ from each other. We use the subscript $n$ to emphasize the dependence on the lattice $(\sZ^d_n,\E^d _n)$. With each edge $e\in\E_n^d $ we associate a random variable $t(e)$ with value in $\sR^+$. The family $(t(e))_{e\in\E_n^d}$ is independent and identically distributed with a common law $G$. Throughout the paper, we work with a distribution $G$ on $\sR^+$ satisfying the following hypothesis.

\noindent{\bf Hypothesis.} The distribution $G$ is such that $G(\{0\})<1-p_c(d)$ and $G$ admits an exponential moment, \textit{i.e.}, there exists $\theta>0$  such that $\int_{\sR^+} \exp(\theta x)dG(x)< \infty$. 

\subsection{Maximal flow}
Let $A$ be a compact convex subset of $\sR^d$. For $x=(x_1,\dots,x_d)$, we define $$\|x\|_2=\sqrt{\sum_{i=1}^dx_i^2}\,.$$  We denote by $\cdot$ the standard scalar product in $\sR^d$. 
A stream is a function $\ff : \E_n^d\rightarrow \sR^d$  such that the vector $\ff(e)$ is colinear with the geometric segment associated with $e$. For $e\in \E_n^d$, $\|\ff(e)\|_2$ represents the amount of water that flows through $e$ per second and $\ff(e)/(n\|\ff(e)\|_2)$ represents the direction in which the water flows through $e$. We say that a stream $\ff$  between $A$ and infinity is admissible if and only if it satisfies the following constraints.
\begin{itemize}
\item[$\cdot$] \textit{The node law} : for every vertex $x$ in $\sZ_n^d\setminus A$, we have
$$\sum_{y\in\sZ_n^d:\,e =\langle x,y \rangle\in \E_n^d } \ff(e)\cdot \overrightarrow{xy} =0\,.$$
\item[$\cdot$] \textit{The capacity constraint}: for every edge $e\in \E_n ^d$, we have 
$$0\leq \|\ff(e)\|_2\leq t(e)\,.$$
\end{itemize}
The node law expresses that there is no loss or creation of fluid outside $A$. The capacity constraint imposes that the amount of water that flows through an edge $e$ per second is limited by its capacity $t(e)$.
As the capacities are random, the set of admissible streams between $A$ and infinity is also random. For each admissible stream $\ff$, we define its flow in the lattice  $(\sZ^d_n,\E^d _n)$ by 
$$\flow(\ff)=\sum_{x\in A\cap \sZ_n^d}\sum_{y\in\sZ^d:\,e =\langle x,y \rangle\in \E_n^d }\ff(e)\cdot \overrightarrow{xy}\,.$$
This corresponds to the amount of water that enters in $\sR^d\setminus A$ through $\partial A$ per second for the stream $\ff$. 
The maximal flow  between $A$ and infinity for the capacities $(t(e))_{e\in\E_n^d}$, denoted by $\phi_n(A\rightarrow \infty)$, is the supremum of the flows of all admissible streams between $A$ and infinity:
$$\phi_n(A\rightarrow \infty)=\sup\left\{\flow(\ff)\,:\,\begin{array}{c}\ff\text{ is an admissible stream between }\\ \text{$A$ and infinity in the lattice  $(\sZ^d_n,\E^d _n)$}\end{array}\right\}\,.$$

\subsection{The max-flow min-cut theorem}
Dealing with admissible streams is not so easy, however we can use an alternative interpretation of the maximal flow which is more convenient. Let $E\subset \E_n^d$ be a set of edges. We say that $E$ separates $A$ from infinity (or is a cutset, for short), if every path from $A$ to infinity goes through an edge in $E$. 
We associate with any set of edges $E$ its capacity $V(E)$ defined by $$V(E)=\sum _{e\in E} t(e)\,.$$ The max-flow min-cut theorem, a classical result of graph theory \cite{Bollobas}, states that 
$$\phi_n(A\rightarrow \infty)=\inf                                                                                              \Big\{\, V(E):\, E \text{ separates $A$ from infinity in $(\sZ_n^d,\E_n^d)$}\, \Big\}\,.$$
We recall that $\mincut_n(A,\infty)$ is the infimum of the  capacities of all cutsets from $A$ to infinity in the lattice $(\sZ_n^d,\E_n^d)$. Note that it is not even obvious whether this infimum is attained. This theorem originally concerns finite graphs but it can be extended to infinite graphs (see for instance section 6.1. in \cite{Garet2}).
We extend the notation $\phi_n$ to any connected subgraph $\cG\subset \sZ_n^d$ and $\fG_1$, $\fG_2$ disjoint subsets of $\cG$:
$$\phi_n(\fG_1\rightarrow \fG_2 \text{ in }\cG)=\inf                                                                                              \Big\{\, V(E):\, E \text{ separates $\fG_1$ from $\fG_2$ in $\cG$}\, \Big\}\,.$$
 
\section{Some notations and useful results}\label{s3}
\subsection{Geometric notations}
Let $S\subset \sR^d$. We define the distance between a point and $S$ by $$\forall x\in\sR^d\,  d_2(x,S)=\inf_{y\in S}\|x-y\|_2$$ and for $r>0$, we define the open $r$-neighborhood $\cV(S,r)$ of $S$ by
$$\cV(S,r)=\Big\{\,x\in\sR^d \,:\, d_2(x,S)< r\,\Big\}\,.$$
Let $x\in \sR^d$, $r>0$ and a unit vector $v$. We denote by $B(x,r)$ the closed ball of radius $r$ centered at $x$, by $\disc(x,r,v)$ the closed disc centered at $x$ of radius $r$ normal to $v$, and by $B^+(x,r,v)$ (respectively $B^- (x,r,v)$) the upper (resp. lower) half part of $B(x,r)$ along the direction of $v$, \textit{i.e.},
$$B^+ (x,r,v)=\Big\{\,y\in B(x,r)\,:\, (y-x)\cdot v\geq 0\,\Big\},$$
and
$$B^- (x,r,v)=\Big\{\,y\in B(x,r)\,:\, (y-x)\cdot v\leq 0\,\Big\}\,.$$ We denote by $\cL ^d$ the $d$-dimensional Lebesgue measure. We denote by $\alpha_d$ the $\cL^d$ measure of a unit ball in $\sR^d$. We denote by $\cH^{d-1}$ the Hausdorff measure of dimension $d-1$. In particular, the $\cH^{d-1}$ measure of a $d-1$ dimensional unit disc in $\sR^d$ is equal to $\alpha_{d-1}$.
 Let $A$ be a non-degenerate hyperrectangle, \textit{i.e.}, a rectangle of dimension $d-1$ in $\sR^d$. Let $\vv$ be one of the two unit vectors normal to $A$. Let $h>0$, we denote by $\cyl(A,h)$ the cylinder of basis $A$ and height $h$ defined by 
$$\cyl(A,h)=\Big\{\,x+t\vv\, : \,  x\in A,\, t\in[-h,h]\,\Big\}\,.$$
The dependence on $\vv$ is implicit in the notation $\cyl(A,h)$.
We also define the infinite cylinder of basis $A$ in a direction $\uu$ (not necessarily normal to $A$):
$$\cyl(A,\uu,\infty)=\Big\{\,x+t\uu\, : \,  x\in A,\, t\geq 0\,\Big\}\,.$$
Note that these definitions of cylinder may be extended in the case where $A$ is a set of linear dimension $d-1$, \textit{i.e.}, $A$ is included in an hyperplane of $\sR^d$, which is the affine span of $A$.
\subsection{Sets of finite perimeter and surface energy}
The perimeter of a Borel set $S$ of $\sR^d$ in an open set $O$ is defined as 
$$\cP(S,O)=\sup \left\{\int _S \dive f(x)d\cL^d(x):\, f\in C^\infty_c(O,B(0,1))\right\}\, $$
where $C^\infty_c(O,B(0,1))$ is the set of the functions of class $\sC^\infty$ from $\sR^d$ to $B(0,1)$ having a compact support included in $O$, and $\dive$ is the usual divergence operator. The perimeter $\cP(S)$ of $S$ is defined as $\cP(S,\sR^d)$. The topological boundary of $S$ is denoted by $\partial S$.  The reduced boundary $\partial ^* S$ of $S$ is a subset of $\partial S$ such that, at each point $x$ of $ \partial ^* S$, it is possible to define a normal vector $n_S(x)$ to $S$ in a measure-theoretic sense, and moreover $\cP(S)=\cH^{d-1}(\partial^*S)$. 
We denote by $\nu$ the flow constant that is a function from the unit sphere $\sS^{d-1}$ of $\sR ^d$ to $\sR ^+$ as defined in \cite{flowconstant}. We  denote by $\nu_{max}$ and $\nu_{min}$ its maximal and minimal values on the sphere. The flow constant $\nu(v)$ corresponds to the expected maximal amount of water that can flow per second in the direction of $v$. A more rigorous definition will be given later.
We can define the associated Wulff crystal $\cW_\nu$:
$$\cW_\nu=\Big\{\,x\in\sR^d\,:\, \forall y, \,\, y\cdot x\leq \nu(y)\,\Big\}.$$ 
With the help of the Wulff crystal, we can define the surface energy of a general set.
\begin{defn} The surface energy $\cI(S,O)$ of a Borel set $S$ of $\sR^d$ in an open set $O$ is defined as 
$$\cI(S,O)=\sup \left\{\int _S \dive f(x)d\cL^d(x):\, f\in C^1_c(O,\cW_\nu)\right\} .$$
\end{defn}
\noindent We will note simply $\cI(S)=\cI(S,\sR^d)$.
\begin{prop}[Proposition 14.3 in \cite{Cerf:StFlour}] The surface energy $\cI(S,O)$ of a Borel set $S$ of $\sR^d$ of finite perimeter in an open set $O$ is equal to
$$\cI(S,O)=\int _{\partial^* S \cap O} \nu(n_S(x))d\cH^{d-1}(x)\,.$$
\end{prop}

We recall the two following fundamental results.
\begin{prop}[Isoperimetric inequality]\label{isop}
There exist two positive constants $b_{iso}$, $c_{iso}$ which depend only on the dimension $d$, such that, for any Cacciopoli set $E$, any ball $B(x,r)\subset \sR^d$,
$$\min\left(\cL^d(E\cap B(x,r)),\cL^d((\sR^d\setminus E)\cap B(x,r))\right)\leq b_{iso}\cP(E,\overset{o}{B}(x,r))^{d/d-1},$$
$$\min\left(\cL^d(E),\cL^d(\sR^d\setminus E)\right)\leq c_{iso} \cP(E)^{d/d-1}\,.$$
\end{prop}

\begin{thm}[Gauss-Green theorem]\label{GGthm}
For any compactly supported $\sC^1$ vector field $f$ from $\sR^d$ to $\sR^d$, any Caccioppoli set $E$,
$$\int_E \dive f(x)d\cL^d(x)=\int_{\partial^* E}f(x)\cdot n_E(x) d\cH ^{d-1}(x)\,.$$
\end{thm}
\subsection{Approximation by convex polytopes}
We recall here an important result, which allows to approximate adequately a set of finite perimeter by a convex polytope.
\begin{defn}[Convex polytope] We say that a subset $P$ of $\sR ^d$ is a convex polytope if there exist $v_1,\dots,v_m$ unit vectors and $\varphi_1,\dots,\varphi_m$ real numbers such that
$$P=\bigcap_{1\leq i\leq m}\Big\{\,x\in \sR^d  : \, x\cdot v_i\leq \varphi_i\,\Big\}\,.$$
We denote by $F_i$ the face of $P$ associated with $v_i$, \textit{i.e.}, $$F_i= P\cap\Big\{\,x\in \sR^d  : \, x\cdot v_i= \varphi_i\,\Big\}\,.$$
\end{defn}
Any compact convex subset of $\sR^d$ can be approximated from the outside and from the inside by a convex polytope with almost the same surface energy.
\begin{lem}\label{ApproP} Let $A$ be a bounded convex set in $\sR ^d$. For each $\ep>0$, there exist convex polytopes $P$ and $Q$ such that $P\subset A\subset Q$ and $\cI(Q)-\ep\leq \cI(A)\leq \cI(P)+\ep$.  
\end{lem}
\begin{proof}
Let $A$ be a bounded convex set in $\sR^d$. Let $\ep>0$. Let $(x_k)_{k\geq 1}$ be a dense family in $\partial A$. For $n\geq 1$, we define $P_n$ as the convex hull of $x_1,\dots,x_n$, \textit{i.e.}, the smallest convex set that contains the points $x_1,\dots,x_n$. 
As $A$ is convex, we have $P_n\subset A$ and $P_n$ converges towards $A$ when $n$ goes to infinity for the $\cL ^1$ topology.
The functional $\cI$ is lower semi-continuous, thus
$$\cI(A)\leq\liminf_{n\rightarrow \infty}\cI(P_n)\,,$$
so there exists $n$ large enough such that 
$$\cI(A)\leq \cI(P_n)+\ep\,$$
and we take $P=P_n$.
The existence of $Q$ was shown by Cerf and Pisztora in Lemma 5.1. in \cite{cerf2000} for the Wulff shape. The proof may be easily adapted to a general convex bounded set $A$. 
\end{proof}

\section{Background on maximal flow}\label{s4}

We now consider two specific maximal flows through a cylinder for first passage percolation on $\sZ^d_n$ where the law of capacities is given by the distribution $G$. We are interested in the cutsets in a cylinder. 
Let us first define the maximal flow from the top to the bottom of a cylinder.
Let $A$ be a non-degenerate hyperrectangle and let $\vv$ be one of the two unit vectors normal to $A$. Let $h\geq 0$. In order to define the flow from the top to the bottom, we have to define discretized versions of the bottom $B_n(A,h)$ and the top $T_n(A,h)$ of the cylinder $\cyl(A,h)$ in the lattice $(\sZ_n^d ,\E_n^d)$. We define 
$$B_n(A,h):= \left\{x\in\sZ_n^d\cap\cyl(A,h)\,:\,\begin{array}{c}
\exists y \notin \cyl(A,h),\, \langle x,y \rangle\in\E_n^d \\\text{ and $\langle x,y \rangle$ intersects } A -h\vv
\end{array} \right\}$$
and
$$T_n(A,h):= \left\{x\in\sZ_n^d\cap\cyl(A,h)\,:\,\begin{array}{c}
\exists y \notin \cyl(A,h),\, \langle x,y \rangle\in\E_n^d \\\text{ and $\langle x,y \rangle$ intersects } A+h\vv
\end{array} \right\}\,.$$
We denote by $\phi_n(A,h)$ the maximal flow from the top to the bottom of the cylinder $\cyl(A,h)$ in the direction $\vv$ in the lattice $(\sZ_n^d,\E_n^d)$, defined by 
$$\phi_n(A,h)=\phi_n\Big(\,T_n(A,h)\rightarrow B_n(A,h) \text{ in } \cyl(A,h)\,\Big)\,.$$
This definition of the flow is not well suited to subadditive arguments, because we cannot glue together two cutsets from the top to the bottom of two adjacent cylinders in order to get a cutset from the top to the bottom of the union of these two cylinders. The reason is that the trace of a cutset from the top to the bottom of a cylinder on the boundary of the cylinder is totally free. We go around this problem by introducing another flow through the cylinder which is genuinely subadditive. The set $\cyl(A,h)\setminus A$ has two connected components, denoted by $C_1(A,h)$ and $C_2(A,h)$. We  define discretized versions of the boundaries of these two sets in the lattice $(\sZ_n ^d,\E_n^d)$. For $i=1,\,2$, we define
$$C'_{i,n}(A,h)=\Big\{\,x\in\sZ_n^d\cap C_i(A,h)\,:\,\exists y\notin \cyl(A,h),\,\langle x,y\rangle\in\E_n^d\,\Big\}\,.$$
We call informally $C'_{i,n}(A,h)$, $i=1,2$, the upper and lower half part of the boundary of $\cyl(A,h)$.
We denote by $\tau_n(A,h)$ the maximal flow from the upper half part to the lower half part of the boundary of the cylinder, \textit{i.e.}, 
$$\tau_n(A,h)=\phi_n\Big(\,C'_{1,n}(A,h)\rightarrow C'_{2,n}(A,h)\text{ in } \cyl(A,h)\,\Big)\,.$$
By the max-flow min-cut theorem, the maximal flow $\tau_n(A,h)$ is equal to the minimal capacity of a set of edges $E\subset \E_n^d$ that cuts $C'_{1,n}(A,h)$ from $C'_{2,n}(A,h)$ inside the cylinder $\cyl(A,h)$. We will need the following upper large deviation result. 
\begin{thm}[Upper large deviations of the maximal flow in a cylinder]\label{upperlargedeviationcyl}
Let $G$ be a probability measure on $[0,+\infty[$ such that $G(\{0\})< 1-p_c(d)$ and $G$ admits an exponential moment, \textit{i.e.}, there exists $\theta>0$  such that $\int_{\sR^+} \exp(\theta x)dG(x)< \infty$. For every unit vector $v$, for every non-degenerate hyperrectangle $A$ normal to $v$, for every  $h>0$ and for every  $\lambda>\nu(v)$, there exist positive real numbers $C_1$ and $C_2$ depending only on $\lambda$ and $G$, such that, for all $n\geq 0$,
$$\Prb\left[ \phi_{n}(A,h)\geq \lambda \cH^{d-1}(A)n^{d-1}\right]\leq C_1\exp(-C_2h n^{d})\,.$$
\end{thm} 
\noindent This theorem is a straightforward application of Theorem 4 in \cite{TheretUpperTau14}.
To ease the reading, constants may change from appearance to appearance.

\section{Upper large deviations} \label{sectionULD}
The goal of this section is to prove Theorem \ref{ULD}.
  \subsection{The case of a cylinder}
  In this section, we will use Theorem \ref{upperlargedeviationcyl} which is the main probabilistic estimate needed to prove Theorem \ref{ULD}.  A convex polytope of dimension $d-1$ is a convex polytope $F$ which is contained in an hyperplane of $\sR^d$ and such that $\cH^{d-1}(F)>0$. We have the following Lemma.
\begin{lem}\label{lem2}Let $F$ be a convex polytope of dimension $d-1$. Let $v$ be a unit vector normal to $F$. Let $h>0$. There exist positive real numbers $C_1$ and $C_2$ depending on $F$, $G$ and $d$ such that, for all $n\geq 1$, for all $\lambda>\nu(v)\cH^{d-1}(F)$,
$$\Prb[\tau_n(F,h)\geq \lambda n^{d-1}]\leq C_1\exp(-C_2 n^{d-1})$$
\end{lem}
\begin{proof} Let $F$ be a convex polytope of dimension $d-1$ and $v$ a unit vector normal to $F$.   We shall cover $F$ by a finite family of hypersquares and control the probability that the flow is abnormally big in $\cyl(F,h)$ by the probability that the flow is abnormally big in one of the cylinders of square basis. Let $\lambda>\nu(v)\cH^{d-1}(F)$.
Let $\kappa>0$ be a real number that we will choose later. We denote by $S(\kappa)$ an hypersquare of dimension $d-1$ of side length $\kappa$ and normal to $v$. We shall cover the following subset of $F$ by hypersquares isometric to $S(\kappa)$:
$$D(\kappa,F)=\Big\{\,x\in F\,: \, d(x,\partial F)>2\sqrt{d}\kappa\,\Big\}\,.$$
There exists a finite family $(S_i)_{i\in I}$ of subsets of $F$, which are translates of $S(\kappa)$ having pairwise disjoint interiors and such that $D(\kappa,F )\subset\cup_{i\in I} S_i$ (see figure \ref{fig2}). Moreover, we have
\begin{align}\label{controleI}
|I|\leq \frac{\cH ^{d-1}(F)}{\cH ^{d-1}(S(\kappa))}\,
\end{align}
and there exists a constant $c_d$ depending only on the dimension such as 
$$\cH^{d-1}\big(F\setminus D(\kappa,F)\big)\leq c_d\cH^{d-2}(\partial F) \,\kappa\,.$$

\begin{figure}[H]
\def\svgwidth{0.5\textwidth}
\begin{center}
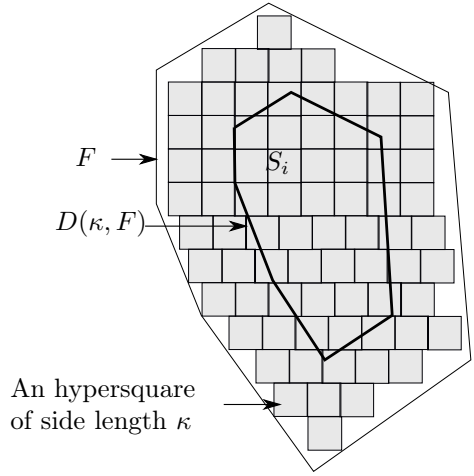
\caption[fig2]{\label{fig2}Covering $P$ with hypersquares}
\end{center}
\end{figure} Let $h>0$.
We want to build a cutset between $C'_1(F,h)$ and $C'_2(F,h)$ out of minimal cutsets for the flows  $\tau_n(S_i,h)$, $i\in I$. Note that a cutset that achieves the infimum defining $\tau_n(S_i,h)$ is anchored near the boundary $\partial S_i$. However, if we pick up two hypersquares $S_i$ and $S_j$ that share a common side, their corresponding minimal cutsets for the flow $\tau_n$ do not necessarily have the same trace on the common face of the associated cylinders $\cyl(S_i,h)$ and $\cyl(S_j,h)$. We shall fix this problem by adding extra edges in order to glue properly the cutsets. Due to the discretization, we will need also to add extra edges around the boundaries of the hypersquares $\partial S_i$ and in the region $F\setminus D(\kappa,F)$ in order to build a cutset.
For $i\in I$, let $E_i$ be a minimal cutset for $\tau_n(S_i,h)$, \textit{i.e.}, $E_i\subset \E_n^d$ separates $C'_1(S_i,h)$ from $C'_2(S_i,h)$ in $\cyl(S_i,h)$ and $V(E_i)=\tau_n(S_i,h)$. 
We fix $\zeta=4d/n$. Let  $E_0$ be the set of the edges of $\E_n^d$ included in $\cE_{0}$, where we define 
$$\cE_0=\Big\{\,x\in\sR^d\,:\,d(x,F\setminus \cup_{i\in I}S_i)\leq \zeta\,\Big\}\cup\bigcup_{i\in I}\Big\{\,x\in\sR^d\,:\,d(x,\partial S_i)\leq \zeta\,\Big\}\,.$$ 
The set of edges $E_0\cup\bigcup_{i\in I} E_i$ separates $C'_1(F,h)$ from $C'_2(F,h)$ in $\cyl(F,h)$ therefore
\begin{align}\label{eq31}
\tau_n(F,h)\leq V(E_0)+\sum_{i\in I} V(E_i)= V(E_0)+\sum_{i\in I}\tau_n(S_i,h)\,.
\end{align}
There exists a constant $c'_d$ depending only on $d$ such that:
\begin{align*}
|E_0|\leq c'_d\left(\kappa n^{d-1}\cH^{d-2}(\partial F)+|I|\cH^{d-2}(\partial S(\kappa))n^{d-2}\right)\,.
\end{align*}
 Using \eqref{controleI}, we obtain
 \begin{align*}
|E_0|&\leq c'_d\left(\kappa n^{d-1}\cH^{d-2}(\partial F)+\dfrac{\cH^{d-1}(F)}{\cH^{d-1}(S(\kappa))}\cH^{d-2}(\partial S(\kappa))n^{d-2}\right)\\
&\leq c'_d\left(\kappa n^{d-1}\cH^{d-2}(\partial F)+2d\dfrac{\cH^{d-1}(F)}{\kappa}n^{d-2}\right)\,.
\end{align*}
Thus, for $n$ large enough, 
\begin{align}\label{eq32}
|E_0| \leq 2c'_d\,\kappa\,\cH^{d-2}(\partial F) n^{d-1}\,.
\end{align}
There exists $s>0$ such that $\lambda>(1+s)\nu(v)\cH ^{d-1}(F)$. Thanks to inequality \eqref{eq31}, we obtain
\begin{align}\label{eqb1}
\Prb[\tau_n(F,h)\geq \lambda n^{d-1}]&\leq\Prb\left[V(E_0)+\sum_{i\in I}\tau_n(S_i,h)\geq (1+s)\nu(v)\cH ^{d-1}(F)n^{d-1}\right]\nonumber\\
&\leq \sum_{i\in I} \Prb[\tau_n(S_i,h)\geq (1+s/2)\nu(v)\cH ^{d-1}(S_i)n ^{d-1}]\nonumber\\
&\hspace{0.5cm} + \Prb\left[V(E_0)\geq \frac{s}{2}\nu(v)n ^{d-1}\cH ^{d-1}(F)\right]\nonumber\\ 
&\leq \sum_{i\in I} \Prb[\phi_n(S_i,h)\geq (1+s/2)\nu(v)\cH ^{d-1}(S_i)n ^{d-1}]\nonumber\\
&\hspace{0.5cm} + \Prb\left[\sum_{i=1}^{2c'_d\kappa\cH^{d-2}(\partial F) n^{d-1}}t_i\geq \frac{s}{2}\nu(v)n ^{d-1}\cH ^{d-1}(F)\right] \,,
\end{align}
where $(t_i)_{i\in \sN}$ is a family of i.i.d. random variables of common probability distribution $G$. We use in the last inequality the fact that $\phi_n(S_i,h)\geq \tau_n(S_i,h)$ and inequality \eqref{eq32}. We can choose $\kappa$ small enough so that
$$2c'_d\kappa\cH^{d-2}(\partial F)  \E[t_i]<s\nu_{min}\cH ^{d-1}(F)/2\, .$$
Moreover, as $G$ admits an exponential moment, the Cram{\'e}r theorem in $\sR$ gives the existence of positive constants $D$ and $D'$ depending on $G$, $F$, $s$ and $d$ such that, for any $n\geq 1$,
\begin{align}\label{eqb2}
\Prb\left[\sum_{i=1}^{2c'_d\kappa\cH^{d-2}(\partial F) n^{d-1}}t_i\geq \frac{s}{2}\nu(v)n ^{d-1}\cH ^{d-1}(F)\right]\leq D\exp(-D'n^{d-1})\,.
\end{align}
Thanks to Theorem \ref{upperlargedeviationcyl}, there exist positive real numbers $C_1$, $C_2$ such that  for $i\in I$, for any $n\geq 1$,
\begin{align}\label{eqb22}
\Prb[\phi_n(S_i,h)\geq (1+s/2)\nu(v)\cH ^{d-1}(S_i)n ^{d-1}]\leq C_1\exp(-C_2hn^d).
\end{align}
By combining inequalities \eqref{eqb1} and \eqref{eqb2} and \eqref{eqb22}, we obtain
\begin{align*}
\Prb[\tau_n(F,h)\geq \lambda n^{d-1}]\leq  D\exp(-D'n^{d-1})+ |I|C_1\exp(-C_2hn^d)\,,
\end{align*}
and the result follows.
\end{proof} 

\subsection{Proof of Theorem \ref{ULD}}
Let $A$ be a compact convex subset of $\sR^d$.   Let $\lambda>\cI(A)$ and let $s>0$ be such that $\lambda>(1+s)\cI(A)$. 
By Lemma \ref{ApproP}, there exists a convex polytope $P$ such that $A\subset P$ and 
\begin{align}\label{condI}
(1+s)\cI(A)\geq (1+s/2)\cI(P)\,.
\end{align}
Let us denote by $F_1,\dots, F_m$ the faces of $P$ and let $v_1,\dots,v_m$ be the associated exterior unit vectors. Let $\ep>0$. For $i\in\{1,\dots,m\}$, we define $C_i=\cyl(F_i+\ep v_i,\ep)$. The sets $C_i$, $1\leq i\leq m$, have pairwise disjoint interiors. Indeed, assume that there exists $z\in \smash{\overset{o}{C_i}\cap \overset{o}{C_j}}$ for some $i\neq j$. Then there exist unique $x\in F_i$, $y\in F_j$ and $h,h'<2\ep$ such that $z=x+hv_i=y+h'v_j$. In fact, the point $x$ (respectively $y$) is the orthogonal projection of $z$ on the face $F_i$ (resp. $F_j$). As $P$ is convex, the orthogonal projection of $z$ on $P$ is unique so $x=y$ and $x\in F_i\cap F_j$. In particular, the point $x$ is in the boundary of $F_i$. This contradicts the fact that $z$ belongs to the interior of $C_i$.
We aim now to build a cutset that cuts $P$ from infinity out of cutsets of minimal capacities for $\tau_n(F_i+\ep v_i, \ep)$, $i\in\{1,\dots,m\}$. The union of these cutsets is not enough to form a cutset from $P$ to infinity because there might be holes between these cutsets. For $i\in\{1,\dots,m\}$, a minimal cutset for $\tau_n(F_i+\ep v_i, \ep)$ is pinned around the boundary of $\partial (F_i+\ep v_i)$. We need to add bridges around $\partial (F_i+\ep v_i)$ to close the potential holes between these cutsets (see figure \ref{fig3}). As the distance between two adjacent $\partial (F_i+\ep v_i)$ decreases with $\ep$, by taking $\ep$ small enough, the size of the bridges and so their capacities is not too big and may be adequately controlled. Next, we shall control the maximal flow through the cylinders or equivalently the capacity of the minimal cutsets in the cylinders thanks to Lemma \ref{lem2}.
 
For $i\in \{1,\dots,m\}$, let $E'_i\subset \E_n  ^d$ be a minimal cutset for $\tau_n(C_i,\ep)$, \textit{i.e.}, $E'_i$ cuts $C'_1(F_i,\ep)$ from $C'_2(F_i,\ep)$ and $V(E'_i)=\tau_n(F_i+\ep v_i,\ep)$. We shall add edges to control the space between $E'_i$ and the boundary $\partial (F_i+\ep v_i)$. 
Let $i,j\in\{1,\dots,m\}$ such that $F_i$ and $F_j$ share a common side. We denote by $\cM(i,j)$:
$$\cM(i,j)=\cV(F_i\cap F_j, \ep +\zeta)\setminus \cV(A,\ep-\zeta)\,.$$
\begin{figure}[H]
\def\svgwidth{0.9\textwidth}
\begin{center}
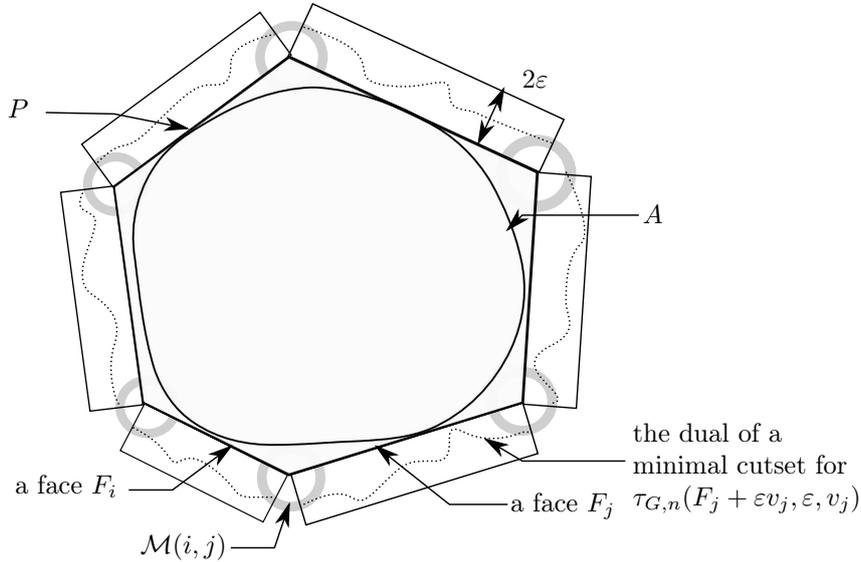
\caption[fig3]{\label{fig3}Construction of a cutset from $P$ to infinity}
\end{center}
\end{figure}
\noindent Let $M_{i,j}$ denote the set of the edges in $\E_n^d$ included in $\cM(i,j)$ (see figure \ref{fig3}).
There exists a constant $c'_d$ depending only on the dimension $d$ such that, for all $i,j\in\{1,\dots,m\}$ such that $F_i$ and $F_j$ share a common side,
\begin{align*}
|M_{i,j}|\leq c_d \ep^{d-1} n^{d-1}\,.
\end{align*}
We set $$M=\bigcup_{i,j}M_{i,j}\,,$$ where the union is over $i,j\in\{1,\dots,m\}$ such that $F_i$ and $F_j$ share a common side. The set $M\cup \left(\bigcup_{i=1}^m E'_i\right)$ cuts $P$ from infinity, therefore
\begin{align}\label{eq41}
\mincut_n(P,\infty) \leq V(M)+\sum_{i=1}^m V(E'_i)= V(M)+\sum_{i=1}^m \tau_n(F_i+\ep v_i,\ep)\,.
\end{align}
As $P$ is a polytope, $$\cI(P)=\sum_{i=1}^m \nu(v_i)\cH^{d-1}(F_i)\,,$$ and as $A\subset P$, we have $\mincut_n(A,\infty)\leq \mincut_n(P,\infty)$. Then, using inequalities \eqref{condI} and \eqref{eq41}, we obtain
\begin{align}\label{eq42}
\Prb&\left[\mincut_n(A,\infty)\geq \lambda n^{d-1}\right]\nonumber\\
&\hspace{3cm}\leq\Prb\left[\mincut_n(A,\infty)\geq (1+s)n^{d-1}\cI(A)\right]\nonumber\\
 &\hspace{3cm}\leq \Prb\left[\mincut_n(P,\infty)\geq (1+s/2)n^{d-1}\cI(P)\right]\nonumber\\
&\hspace{3cm}\leq\Prb\left[V(M)+\sum_{i=1}^m \tau_n(F_i+\ep v_i,\ep)\geq (1+s/2)n ^{d-1}\cI(P)\right]\nonumber\\
&\hspace{3cm}\leq \Prb[V(M)>s\cI(P)n ^{d-1}/4]\nonumber\\
&\hspace{3.5cm}+\sum_{i=1}^m\Prb\left[\tau_n(F_i+\ep v_i,\ep)\geq (1+s/4)n ^{d-1}\cI(F_i)\nu(v_i)\right]\,.
\end{align}
Moreover, we have
\begin{align}\label{eq43}
\Prb\left[V(M)>s\cI(P)n ^{d-1}/4\right]\leq \Prb\left[\sum_{i=1}^{c_d m^2 \ep^{d-1}n ^{d-1}}t_i\geq s\cI(P)n ^{d-1}/4\right]\,,
\end{align}
where $(t_i)_{i\in \sN}$ is a family of i.i.d. random variables of common probability distribution $G$. We choose $\ep$ small enough so that
$$c_d m ^2\ep^{d-1}\E[t_i] <s\cI(P)/4\, .$$
Since $G$ admits an exponential moment, then the Cram{\'e}r theorem in $\sR$ gives the existence of positive constants $D$ and $D'$ depending on $G$, $P$, $s$ and $d$ such that
\begin{align}\label{eq44}
\Prb\left[\sum_{i=1}^{c_d m^2 \ep^{d-1}n ^{d-1}}t_i\geq s\cI(P)/2n ^{d-1}\right]\leq D\exp(-D'n^{d-1})\,.
\end{align}
By Lemma \ref{lem2}, there exist positive real numbers $C_1$ and $C_2$ depending on $P$, $s$, $G$ and $d$ such that for all $i\in\{1,\dots,m\}$,
\begin{align}\label{eq45}
\Prb\left[\tau_n(F_i+\ep v_i,\ep)\geq (1+s/4)n ^{d-1}\cI(F_i)\nu(v_i)\right]\leq C_1\exp(-C_2n^{d-1})\,.
\end{align}
We conclude by combining inequalities \eqref{eq42}, \eqref{eq43}, \eqref{eq44} and \eqref{eq45} that
$$\Prb[\mincut_n(A,\infty)\geq \lambda n^{d-1}]\leq D\exp(-D'n^{d-1})+m \,C_1\exp(C_2 n^{d-1})\,.$$
This yields the desired conclusion.
\section{Existence of a minimal cutset}\label{s6}
In this section, we recall the fundamental result of Zhang which enables to control the cardinality of one specific cutset. We use this opportunity to precise an important point, namely we prove in addition that there exists a minimal cutset $E$ between a convex $A$ and infinity.  We here prove Theorem \ref{thmZhang} using the work of Zhang \cite{Zhang2017}.

Throughout the proof we work on the lattice $(\sZ^d,\E^d)$. Let $A$ be a convex compact subset of $\sR^d$. As any path from $A$ to infinity has to go through an edge of $\partial_e A$, the set $\partial_e A$ cuts $A$ from infinity and $\mincut(A,\infty)\leq V(\partial_e A)$. Let $E$ be a cutset between $A$ and infinity such that $V(E) \leq V(\partial_e A)$.
We want to control the probability that $E$ has too many edges. To do that we distinguish three types of edges that we will handle differently. Let $\ep$ be a positive constant that we will adjust later. We define:
 
$\bullet$ The $\ep^+$ edges are the edges $e\in E$ such that $t(e)>\ep$. We denote by $N^+ (E)$ the number of $\ep^+$ edges in $E$. We can control $N^+ (E)$ thanks to the fact that $V(E)\leq V(\partial_e A)$ and so $\ep N ^+ (E)\leq V(\partial_e A)$.

$\bullet$ The $\ep^-$ edges that are the edges $e\in E$ such that $0<t(e)\leq\ep$. We denote by $N^- (E)$ the number of $\ep^-$ edges in $E$. As the probability of being an $\ep^-$ edge goes to $0$ when $\ep$ goes to $0$, we can choose $\ep$ so that, with high probability, $N^- (E)$ does not exceed a certain proportion of $|E|$, the number of edges in $E$.

$\bullet$ The closed edges or zero edges that are the edges of null passage times. Once we have controlled the number of $\ep^+$ and $\ep ^-$ edges, the size of $E$ cannot be too big otherwise the number of closed edges, would be also big and this would mean that there exist large surfaces of closed edges which is an unlikely event when $G(\{0\})<1-p_c(d)$. 


We start now with these estimates. Let $n\geq 1$. Let $E$ be a cutset from $A$ to infinity such that $V(E)\leq V(\partial_e A)$ and $|E|=n$. We start by controlling the $\ep ^+ $ edges by controlling the capacity of $\partial_e A$:
\begin{align}\label{eq51}
&\Prb\Big(\,\exists E\subset \E^d, \text{ $E$ cuts $A$ from infinity, $V(E)\leq V(\partial_e A)$ and $|E|=n$}\,\Big)\nonumber\\
&\hspace{0.5cm} \leq \Prb\Big(\,\exists E\subset \E^d, \text{ $E$ cuts $A$ from infinity, $V(E)\leq \ep^2n$ and $|E|=n$}\,\Big)\nonumber\\
&\hspace{1cm}+\Prb\Big(\,V(\partial_e A)\geq \ep^2n\,\Big)\nonumber\\
&\hspace{0.5cm} \leq \Prb\Big(\,\exists E\subset \E^d, \text{ $E$ cuts $A$ from infinity, $N^+(E)\leq \ep n$ and $|E|=n$}\,\Big)\nonumber\\
&\hspace{1cm}+\Prb\Big(\,V(\partial_e A)\geq \ep^2n\,\Big)\,.
\end{align}
As $G$ admits an exponential moment, we obtain
\begin{align}\label{corre1}
\Prb\Big(\,V(\partial_e A)\geq \ep^2n\,\Big)&=\Prb\Big(\,\theta V(\partial_e A)\geq \theta \ep^2n\,\Big)\nonumber\\
&\leq \exp(-\theta\ep^2 n)\E(\exp(\theta V(\partial_e A)))\nonumber\\
&= \exp(-\theta\ep^2 n)\left(\int_{\sR^+}\exp(\theta x)dG(x)\right)^{|\partial_e A|}\,.
\end{align}
We take $$\lambda=2\ln \left(\int_{\sR^+}\exp(\theta x)dG(x)\right)/\theta\ep^2\,.$$ For $n>\lambda |\partial_e A|$, we have, using \eqref{corre1},
\begin{align}\label{corre2}
\Prb\Big(\,V(\partial_e A)\geq \ep^2n\,\Big)\leq \exp\left(-\frac{1}{2}\theta \ep ^2 n\right)\,.
\end{align}
Combining inequalities \eqref{eq51} and \eqref{corre2}, we get 
\begin{align}\label{star}
&\Prb\Big(\,\exists E\subset \E^d, \text{ $E$ cuts $A$ from infinity, $V(E)\leq V(\partial_e A)$ and $|E|=n$}\,\Big)\nonumber\\
&\hspace{1cm} \leq \Prb\Big(\,\exists E\subset \E^d, \text{ $E$ cuts $A$ from infinity, $N^+(E)\leq \ep n$ and $|E|=n$}\,\Big)\nonumber\\
&\hspace{1.5cm}+\exp\left(-\frac{1}{2}\theta \ep ^2 n\right)\,.
\end{align}
We control next the number of $\ep ^- $ edges. We define $\delta_1=\delta_1(\ep)=G(]0,\ep])$ the probability that an edge $e$ is an $\ep ^-$ edge.   The probability $\delta_1(\ep)$ goes to $0$ when $\ep$ goes to $0$.
We bound the number of cutsets of size $n$ with the help of combinatorial arguments. 
As in the original proof of Zhang, we fix a vertex belonging to an edge of $E$. Since $E$ is a cutset, then at least one edge of $E$ has an extremity on the vertical line $L=\{\,(0,\dots,0,x_d), \,x_d\in\sR\,\}$. Moreover, the set $E$ is finite. Let $z=(0,\dots,0,x_d)$ be the highest vertex of $L$ belonging to an extremity of an edge of $E$.  Since $|E|\leq n$, then certainly $x_d\leq n$. 
We denote by $\widehat{E}$ the set of the vertices of $\sZ^d$ that are connected to a vertex in $A$ without using an edge in $E$, \textit{i.e.},
$$\widehat{E}=\left\{x\in \sZ^d:\,\,\begin{array}{c}
\text{there exists a path from $x$ to $A$ which }\\\text{does not go through an edge in $E$}\end{array} \right\}$$
We denote by $\partial_v \widehat{E}$ the exterior vertex boundary of $\widehat{E}$, defined as
$$\partial_v \widehat{E}=\left\{x\in \sZ^d\setminus \widehat{E}:\begin{array}{c} \text{ $x$ has a neighbour in $\widehat{E}$ and there exists}\\\text{a path from $x$ to infinity in $\sR^d\setminus \widehat{E}$}\end{array}\right\}\,.$$
This set is the analogue of $\partial_e \widehat{Z}(k,m)$ in \cite{Zhang2017}.
By Lemma 10 in \cite{Zhang2017}, the set $\partial_v \widehat{E}$ is $\sZ^d$ connected, it contains $z$ and moreover 
$$|\partial_v \widehat{E}|\leq 3^{d+1}n\,.$$
Once the vertex $z$ is fixed, the set $\partial_v \widehat{E}$ is a $\sZ^d$ connected set and we can apply the bound (4.24) in \cite{grimmettt:percolation}, there are at most $7^{d3^{d+1}n}$ possible choices for $\partial_v\widehat{E}$. We recall that each vertex has at most $2d$ adjacent edges. Once the set $\partial_v \widehat{E}$ is fixed, we bound the number of possible choices for the set $E$ by 
$$\sum_{k=1}^{3^{d+1}n} \binom{3^{d+1}n}{k}(2d) ^k \leq (2d+1)^{3^{d+1}n}\,.$$
Let $D$ be a positive constant that will be adjusted later. By summing on the coordinate $x_d$ of $z$, on the choice of $\partial_v E$ and $E$, we have
\begin{align}\label{eq52}
&\Prb[\exists E\subset \E^d, \text{ $E$ cuts $A$ from infinity, $N^- (E)\geq -(D/\ln \delta_1 )|E|$ and $|E|=n$}]\nonumber\\
&\hspace{1cm} \leq \sum_{i=0}^n\Prb\left[\begin{array}{c}\exists E\subset \E^d, \text{ $E$ cuts $A$ from infinity,}\\ \text{$N^- (E)\geq -(D/\ln\delta_1)n$, $x_d=i$ and $|E|=n$}\end{array}\right]\nonumber\\
&\hspace{1cm} \leq n 7^{d3^{d+1}n}(2d+1)^ {3^{d+1}n} \max_{\Gamma}\Prb\left[|\Gamma|=n,\, N^-(\Gamma)\geq -Dn/\ln\delta_1\right]\,,
\end{align}
where the maximum is over all the cutsets $\Gamma$ from $A$ to infinity with $n$ edges.
For $\delta_1$ small enough and $D$ large enough, depending only on the dimension $d$, we have
\begin{align}\label{eq53}
\Prb\left[|\Gamma|=n,\, N^-(\Gamma)\geq -Dn/\ln\delta_1\right]\leq 2\exp(-Dn/2)\,.
\end{align}
We refer to the proof of Theorem 1 in \cite{Zhang2017} for the proof of this result. Thus, by taking $\delta_1$ small enough and $D$ large enough and combining \eqref{eq52} and \eqref{eq53}, there exist two constants $C_1$ and $C_2$ depending only on $G$, $d$ and $\delta_1$ such that
\begin{align}\label{eq54}
&\Prb[\exists E\subset \E^d, \text{ $E$ cuts $A$ from infinity, $N^- (E)\geq -(D/\ln\delta_1|E|$ and $|E|=n$}]\nonumber\\
&\hspace{0.5cm}\leq C_1\exp(-C_2 n)\,.
\end{align}
Finally, combining inequalities \eqref{star} and \eqref{eq54}, we obtain
\begin{align}\label{eq55}
&\Prb[\exists E\subset \E^d, \text{ $E$ cuts $A$ from infinity, $V(E)\leq V(\partial_e A)$ and $|E|=n$}]\nonumber\\
&\hspace{0.5cm} \leq \Prb\left[\begin{array}{c}\exists E\subset \E^d, \text{ $E$ cuts $A$ from infinity, $|E|=n$, }\\\text{$N^+(E)\leq \ep n$  and $ N^- (E)\leq -(Dn)/\ln\delta_1$ }\end{array}\right]+C_1\exp(-C_2n)\nonumber\\
&\hspace{1cm}+\exp\left(-\frac{1}{2}\theta \ep ^2 n\right)
\end{align}
We have controlled the numbers of $\ep^+$ edges and $\ep ^- $ edges in the cut. We have now to control the number of closed edges in the cut. We denote by $J$ the number of edges in $E$ of positive capacities. On the event $$\Big\{\,|E|=n, \, N^+(E)\leq \ep n, \,N^- (E)\leq -(Dn)/\ln\delta_1\,\Big\}\,,$$ we have
\begin{align}\label{star2}
J\leq N^+(E)+N^-(E)\leq \left(\ep -D/\ln\delta_1\right)n\,.
\end{align}
Thanks to inequalities \eqref{eq55} and \eqref{star2}, we obtain for $n \geq \lambda|\partial_e A|$,
\begin{align*}
&\Prb\left[\exists E\subset \E^d, \text{ $E$ cuts $A$ from infinity, $V(E)\leq V(\partial_e A)$ and $|E|=n$}\right]\\
&\hspace{0.5cm}\leq \Prb\left[\exists E\subset \E^d, \text{ $E$ cuts $A$ from infinity, $|E|=n$ and $J\leq \left(\ep -D/\ln\delta_1\right)n$}\right]\\
&\hspace{1cm}+C_1\exp(-C_2n)+\exp\left(-\frac{1}{2}\theta \ep ^2 n\right)\,.
\end{align*}
The remaining of the proof consists in controlling the zero edges. We will not write the details but only sketch the main ideas of the control. We say that an edge is closed if it has null capacity, otherwise we say that the edge is open.
Let us consider the set $\sC(A)$ that contains all the vertices that are connected to $A$ by an open path. On the event that there exists a cutset of null capacity that cuts $A$ to infinity, the set $\sC(A)$ is finite and its edge boundary $\partial_e\sC(A)$ is a cutset of null capacity. However, this cutset may be very tangled and may contain too many  edges. From this cutset, we want to build a "smoother" cutset, which has smaller cardinality. We use a renormalization procedure at a scale $t$ (which is defined later), and we exhibit a set of boxes $\Gamma_t$ that contains a cutset of null capacity and such that each box of $\Gamma_t$ has at least one $*$-neighbor in which an atypical event occurs (an event of probability that goes to $0$ when $t$ goes to infinity). As these events are atypical, it is unlikely that $\Gamma_t$ contains too many boxes. 

As we are in a supercritical Bernoulli percolation, \textit{i.e.}, $G(\{0\})<1-p_c(d)$, it is very unlikely that a cutset from $A$ to infinity has null capacity and that $\sC(A)$ is finite. To achieve the construction of $\Gamma_t$, we modify the configuration $\omega$. We first choose $\ep$ small enough such that $$J\leq \left(\ep -D/\ln\delta_1\right)n\leq \frac{n}{(2(36dt))^{3d}}\,.$$ For the edges $e_1,\dots,e_J$ in $E$ such that $t(e_i)>0$, we modify $\omega$ by setting $t(e_i)=0$ for $i\in\{1,\dots,J\}$. This modification of $\omega$ is only formal, it is a trick to build $\Gamma_t$. Later we will switch back the capacities to their original values, the boxes of $\Gamma_t$ that does not contain any $e_1,\dots,e_J$ remain unchanged, yet atypical events still occur in the vicinity of these boxes. The number of boxes in $\Gamma_t$ that have changed when we switch back to the configuration $\omega$ is bounded by the number of edges $J$ that we have closed. We obtain an upper bound on $|\Gamma_t|$ with the help of Peirls estimates on the number of boxes where an atypical event occurs. We finally control the probability that there exists a cutset of size $n$ with $J \leq n/(2(36dt))^{3d}$. These tricky computations are detailed in Zhang's paper \cite{Zhang2017}, so we do not reproduce them here. In the end, we obtain the following estimate: there exist constants $C''_1$ and $C''_2$ depending on $\ep$, $A$ and $G$ such that 
$$\Prb\left[\begin{array}{c}\exists E\subset \E^d, \text{ $E$ cuts $A$ from infinity,}\\\text{ $V(E)\leq V(\partial_e A)$ and $|E|=n$}\end{array}\right] \leq C'_1\exp(-C'_2n)\,.$$
By the Borel-Cantelli Lemma, we conclude that, for $n$ large enough, there does not exist any cutset $E$ from $A$ to infinity of size larger than $n$ and such that $V(E)\leq V(\partial_e A)$. Thus, there exists almost surely a minimal cutset from $A$ to infinity and for $n \geq\lambda |\partial_e A|$,
\begin{align*}
&\Prb[\exists E \subset \E^d,\text{ E is a minimal cutset from $A$ to infinity and $|E|\geq  n$}]\\
&\hspace{1cm}\leq \sum _{k= n}^ \infty\Prb\left[\begin{array}{c}\exists E\subset \E^d, \text{ $E$ cuts $A$ from infinity,}\\\text{ $V(E)\leq V(\partial_e A)$ and $|E|=k$}\end{array}\right]\leq C\exp(-C'n)
\end{align*}
where $C,C'$ are positive constants depending only on $G$, $A$ and $d$.

\section {Lower large deviations}\label{sectionLLD}

In this section we prove Theorem \ref{LLD}.
If $\phi_A=0$, we do not have to study the lower large deviations. We suppose that $\phi_A>0$. Let $\lambda<\phi_A$.
We denote by $\cE_n\subset \E^d_n$ a cutset from $A$ to infinity of minimal capacity, \textit{i.e.}, $V(\cE_n)=\mincut_n(A,\infty)$ and having minimal cardinality (if there is more than one such set we pick one according to a deterministic rule). The existence of such a cut is ensured by Theorem \ref{thmZhang}. The aim of this section is to bound from above the probability $\Prb[V(\cE_n)\leq \lambda n^{d-1}]$.


With high probability, the cut $\cE_n$ does not have too many edges. In the lattice $(\sZ_n^d,\E_n^d)$, the cardinality of $\partial_e A $ is of order $n^{d-1}$, and by applying Theorem \ref{thmZhang}, we obtain the existence of constants $\beta$, $C_1$ and $C_2$ depending on $A$, $G$ and $d$ such that
$$\Prb\left[|\cE_n|\geq \beta n ^{d-1}\right]\leq C_1\exp(-C_2 n ^{d-1})\,.$$
  In the proof, we will use the relative isoperimetric inequality in $\sR^d$. To do so, we define continuous versions of the discrete random sets. We define the set $\widetilde{E_n}\subset \sZ_n^d$ by
$$\widetilde{E_n}=\left\{x\in \sZ^d_n\setminus A_n\,:\,\begin{array}{c}
\text{there exists a path from $x$ to $A$ visiting}\\\text{ only edges that are not in $\cE_n$}\end{array} \right\}\,.$$
Let $C$ be the unit cube in $\sR^d$. We define a continuous version $E_n$ of $\widetilde{E_n}$ by
$$E_n=\bigcup_{x\in\widetilde{E_n}}\left(x+\frac{C}{n}\right)\setminus A.$$
If $|\cE_n|\leq \beta n^{d-1}$ then $\cP(E_n,\sR^d \setminus A)\leq \beta$ and $\cP(E_n)\leq \beta +\cP(A)$.
Moreover if $|\cE_n|\leq \beta n^{d-1}$ then $E_n\subset \cV(A, 2d\,\beta n^{d-2})$. 

The set $E_n$ is a random bounded subset of $\sR^d$. However, the diameter of $E_n$ might be very large, of polynomial order in $n$, and there is no compact region of $\sR^d$ that almost surely contains $E_n$. Therefore, we cannot proceed as in \cite{CerfTheret09infc}. However, as the capacity of $\cE_n$ is small, we expect it to remain close to the boundary of $\partial A$. As moving too far away from $\partial A$ is too expensive for $\cE_n$, we should observe unlikely events just by inspecting what happens near the boundary of $A$. 
Let $R$ be a real number we will choose later such that $A\subset B(0,R)$. We set $$\Omega=\overset{o}{B}(0,R)\cap A^c\,.$$
Note that the set $\Omega$ is open. In the following, we will only work with the portion of $\cE_n$ in $\Omega$.
For $F$ a Borel subset of $\sR^d$ such that $\cP(F,\sR)<\infty$, we define $$ \cI_\Omega(F)= \int_{\partial^* F \cap \Omega}\nu(n_F(x))d\cH^{d-1}(x)+\int_{\partial^* A \cap \partial^* (\Omega \setminus F)}\nu(n_A(x))d\cH^{d-1}(x).$$
The quantity $\cI_\Omega(F)$ may be interpreted as the capacity of the subset $F\cup A$ in $\Omega$.
By definition, we know that $\phi_A\leq \cI(A\cup F)$ but it is not easy to compare $\phi_A$ with $\cI_\Omega(F)$ because $\cI_\Omega(F)$ does not take into account the capacity of $\partial F\setminus \Omega$. In other words, the capacity in $\Omega\cup A$ does not coincide with the capacity in $\sR^d$. To go around this problem, we shall remove some regions of $F$ in the neighborhood of $\partial \Omega$, thereby obtaining a new set $\widetilde{F}$, whose closure is included in $\Omega$, and which therefore satisfies $ \cI(\widetilde{F})=\cI_\Omega(\widetilde{F})$. The delicate point is to build the set $\widetilde{F}$ in such a way that $\smash{\cI(\widetilde{F})}$ is only slightly larger than $\cI_\Omega(F)$. We will perform a geometrical surgery by choosing cutting surfaces which do not create too much extra perimeter.

We introduce the space 
$$\sC_\beta =\Big\{\,F\text{ Borel subset of } \Omega\,:\, \cP(F,\Omega)\leq \beta\,\Big\}$$
endowed with the topology $L^1$ associated to the distance $d(F,F')=\cL^d(F\Delta F')$, where $\Delta $ is the symmetric difference between sets. For this topology, the space $\sC_\beta$ is compact. Let us set $$\overline{E}_n=E_n\cap \Omega\,.$$ The set $\overline{E}_n$ belongs to $\sC_\beta$. Suppose that we associate to each $F\in\sC_\beta$ a positive number $\ep_F$. The collection of open sets 
$$\Big\{\,H\text{ Borel subset of } \Omega\,:\,\cL^d(H\Delta F)<\ep_F\,\Big\},\,F\in\sC_\beta,$$ is then an open covering of $\sC_\beta$. By compactness, we can extract a finite covering $(F_i,\ep_{F_i})_{1\leq i\leq N}$ of $\sC_\beta$. This compactness argument enables us to localize the random set $\overline{E}_n$ near a fixed set $F_i$ of $\sC_\beta$. The number $\ep_F$ associated to $F$ will depend on the set $F$. We will explain later in the proof how it is chosen. For the time being, we start the argument with a covering $(F,\ep_F)$ of $\sC_\beta$.
Let $\delta>0$ be a real number to be adjusted later.
To be able to operate the geometrical surgery, we will localize a region of $\Omega$ that contains a volume of $\overline{E}_n$ less than $\delta$.
As $A$ is compact, there exists a real number $\rho>0$ such that 
\begin{align}\label{corre3}
A\subset \overset{o}{B}(0,\rho)\text{   and   }\cL^d(B(0,\rho))\geq 3 c_{iso} (\cP(A)+\beta)^\frac{d}{d-1}\,.
\end{align}
Moreover, using Proposition \ref{isop}, we get
\begin{align}\label{contvol}
\cL^d(E_n)\leq \cL^d(E_n\cup A)\leq b_{iso}(\cP(A)+\beta)^{d/d-1}\,.
\end{align} 
Let us define for $i\geq 0$ the $i$-th annulus $\cA_i$:
$$\cA_i=B(0,\rho+i+1)\setminus B(0,\rho+i)\,.$$
We also define 
$$\textbf{i}=\min\Big\{\,i\geq 1\,:\, \cL^d(E_n\cap \cA_i)\leq \delta\,\Big\}\,.$$
We write $\textbf{i}$ in bold to emphasize that it is a random index. Thanks to inequality \eqref{contvol}, we obtain
$$\textbf{i}\leq b_{iso}(\cP(A)+\beta)^{d/d-1}/ \delta$$
and the minimum in the definition of $\textbf{i}$ is always attained. 
We set $$M= b_{iso}(\cP(A)+\beta)^{d/d-1}/ \delta\,$$
and $$R=\rho+1+M\,.$$ 
Thus, the region $\cA_{\textbf{i}}$ is included in $\Omega$ and contains a volume of $\overline{E}_n$ less than $\delta$.
We sum over $(F_i,\ep_{F_i})_{1\leq i\leq N}$ of $\sC_\beta$ and condition on $\textbf{i}$ and we get
\begin{align}\label{eqE}
\Prb&[V(\cE_n)\leq \lambda n^{d-1}]\nonumber\\
&\hspace{0.5cm}\leq \Prb[|\cE_n|\geq \beta n^{d-1}]+\Prb[V(\cE_n)\leq \lambda n^{d-1}, \, |\cE_n|\leq \beta n^{d-1}]\nonumber\\
&\hspace{0.5cm}\leq C_1\exp(-C_2\beta n^{d-1})+ \sum_{i=1}^N\sum_{1\leq j\leq M}  \Prb\left[\begin{array}{c}\cL^d(\overline{E}_n\Delta F_i)\leq \ep_{F_i},\\ V(\cE_n)\leq \lambda n^{d-1},\, \textbf{i}=j\end{array}\right]\,,
\end{align}
We control next the probability inside the sums for a generic $F$ in $\sC_\beta$ and for $j$ a value for the random set $\textbf{i}$ which occurs with positive probability. 
By definition of $\textbf{i}$, we have
$$\cL^d(E_n\cap \cA_j)\leq \delta\,.$$
\begin{figure}[H]
\def\svgwidth{0.7\textwidth}
\begin{center}
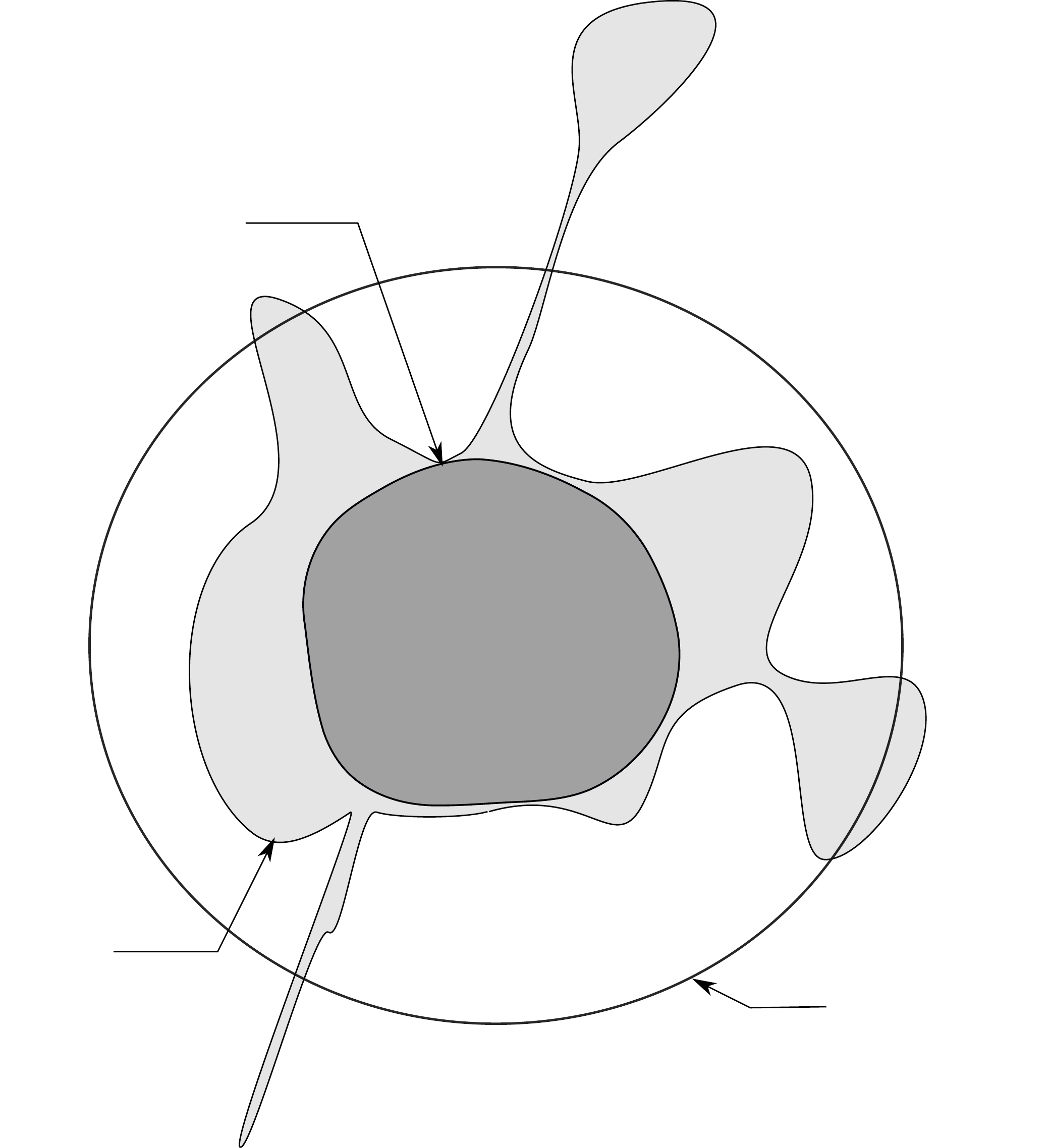
\caption[fig4]{\label{fig4}The set $E_n$ and its associated $\cA_{\textbf{i}}$}
\end{center}
\end{figure}
We want to build from $F$ a set $\widetilde{F}$ of finite perimeter such that its boundary $\partial \widetilde{F}\setminus A$ is in $\Omega$ and $\cI_\Omega (\widetilde{F})$ is close to $\cI_\Omega (F)$. Of course, cutting $F$ inside $\Omega$ creates some extra capacity along the cutting. The idea is to cut $F$ in the annulus $\cA_j$. As the volume of $F$ in this region is small, we shall be able to find cutting surfaces having small perimeter.
If we choose $\ep_F$ small enough such that $\ep_F\leq \delta$ for all $F\in\sC_\beta$, it follows that
$$\cL^d(F\cap \cA_j)\leq \delta +\ep_F\leq 2\delta,.$$
By Lemma 14.4 in \cite{Cerf:StFlour}, for $i\in J$, for $\cH^1$ almost all $t$ in $]0,1[$,
\begin{align}\label{eq1}
&\cI (F\cap B(0,\rho+j+t))\nonumber\\
&\hspace{1.5cm}\leq \cI (F\cap \overset{o}{B}(0,\rho+j+t))+ \nu_{max}\cH^{d-1}(F\cap \partial B(0,\rho+j+t))\,.
\end{align}
Let $T$ be the subset of $]0,1[$ for which the above inequality holds. We have $\cH^1 (T)=1$. Integrating in polar coordinates, we get
\begin{align*}
\int_T \cH^{d-1}(F\cap \partial B(0,\rho+j+t)) 
&= \cL ^d(F\cap B(0,\rho+j+1)\setminus B(0,\rho+j)) \\
&\leq\cL^d(F\cap \cA_j) \leq 2\delta\,.
\end{align*}
Thus, there exists $t\in T$ such that 
\begin{align}\label{eq3}
 \cH^{d-1}(F\cap \partial B(0,\rho+j+t)) \leq 3\delta\,. 
\end{align}
We next define $$\widetilde{F}= F\cap B(0,\rho+j+t)\,.$$
By construction, we have $\partial \widetilde{F} \setminus A \subset \Omega$. Combining inequalities \eqref{eq1} and \eqref{eq3}, we obtain
\begin{align}\label{star4}
\phi_A\leq \cI_\Omega (\widetilde{F}) \leq \cI\big(F,\overset{o}{B}(0,\rho+j+t)\big)+ \nu_{max} 3\delta\leq \cI_\Omega (F)+ 3\delta \nu_{max}
\end{align}
We show next that is possible to choose $\delta$ such that, uniformly over $F$, we have $s\cI_\Omega (F)\geq 3\delta \nu_{max}$.
We have 
\begin{align*}
\cI_\Omega(F)&\geq \int_{ \partial^* (F\cup A)\cap \overset{o}{B}(0,\rho)}\nu(n_{A\cup F}(x))d\cH^{d-1}(x)\geq \nu_{min}\cP\left(F\cup A, \overset{o}{B}(0,\rho)\right)\,.\end{align*}
We apply the isoperimetric inequality relative to the ball $B(0,\rho)$:
\begin{align*}
&\cP\left(F\cup A, \overset{o}{B}(0,\rho)\right)\\
&\hspace{1.5cm}\geq  \left(\frac{\min\left(\cL ^d((A\cup F) \cap B(0,\rho)),\,\cL^d((\sR^d\setminus (A\cup F))\cap B(0,\rho))\right)}{b_{iso}}\right) ^\frac{d-1}{d}\,.
\end{align*}
Since $F$ is in $\sC_\beta$, we have $\cL^d(A\cup F)\leq c_{iso}(\cP(A)+\beta)^\frac{d}{d-1}$.
Together with inequality \eqref{corre3}, we conclude that 
$$\cP\left(F\cup A, \overset{o}{B}(0,\rho)\right)\geq  \left(\frac{\cL ^d(A)}{b_{iso}}\right) ^\frac{d-1}{d}\,.$$
There exists $s>0$ such that $\lambda\leq (1-s)\phi_A$. We choose $\delta$ such that 
\begin{align}\label{choicedelta}
2\delta \nu_{max}=s \nu_{min} \left(\frac{\cL ^d(A)}{b_{iso}}\right) ^\frac{d-1}{d}\,.
\end{align} 
Using inequality \eqref{star4}, we have then, for any $F$ in $\sC_\beta$,
$$ \cL^d(\overline{E}_n\Delta F)\leq \delta\implies s\cI_{\Omega}(F)\geq 3\delta \nu_{max}\implies \lambda \leq (1-s)\phi_A\leq (1-s^2)\cI_\Omega (F)\,.$$
So we get, 
\begin{align}\label{eqF}
&\Prb[V(\cE_n)\leq \lambda n^{d-1},\,\cL^d(\overline{E}_n\Delta F)\leq \ep_{F},\, \textbf{i}=j]\nonumber\\
&\hspace{1cm}\leq  \Prb[V(\cE_n\cap \Omega)\leq (1-s^2)\cI_\Omega(F) n^{d-1},\,\cL^d(\overline{E}_n\Delta F)\leq \ep_{F}\}]\,.
\end{align}
The remaining of the proof follows the same ideas as in \cite{CerfTheret09infc}. We study the quantity $$\Prb[V(\cE_n\cap \Omega)\leq (1-s^2)\cI_\Omega(F) n^{d-1},\,\cL^d(\overline{E}_n\Delta F)\leq \ep_{F}\}]$$ for a generic $F$ in $\sC_\beta$ and its corresponding $\ep_F$. We will need the following lemma to cover $F$ by balls of small radius such that $\partial F$ is "almost flat" in each ball. This lemma is purely geometric, the covering depends only on the set $F$.

\begin{lem}\label{lemball} [Lemma 1 in \cite{CerfTheret09infc}]
Let $F$ be a subset of $\Omega$ of finite perimeter such that $\partial F\cap \partial \Omega=\emptyset$. For every positive constants $\delta'$ and $\eta'$, there exists a finite family of closed disjoint balls $(B(x_i,\rho_i))_{i\in I \cup K}$ and vectors $(v_i)_{i\in I \cup K}$, such that, 
$$\forall i\in I, \, x_i\in \partial^* F \cap \Omega, \, \rho_i\in ]0,1[, \, B_i\subset \Omega\setminus A,\, \cL^d((F\cap B_i)\Delta B_i^-)\leq \delta' \alpha_d \rho_i^d,$$
and letting $B_i=B(x_i,\rho_i)$ and $B_i^-=B^-(x_i,\rho_i,v_i)$, we have
\begin{align*}\forall i\in K, \, x_i\in \partial^* A \cap \partial^*(\Omega\setminus F), \, \rho_i\in ]0,1[, \, \partial \Omega\cap B_i& \subset \partial ^* A\setminus \partial^* F,\, &\\ \cL^d((A\cap B_i)\Delta B_i^-)\leq \delta' \alpha_d \rho_i^d,\end{align*}
and finally 
$$\left|\cI_\Omega(F)- \sum_{i\in I} \alpha_{d-1}\rho_i^{d-1}(\nu(n_F(x_i))-\sum_{i\in K} \alpha_{d-1}\rho_i^{d-1}(\nu(n_A(x_i))\right |\leq \eta. $$
\end{lem}
\noindent First notice that $$\phi_A\leq \int _{\partial^*A}\nu(n_A(x))d\cH^{d-1}(x)<\infty\,.$$ We choose $\eta = s^4 \cI_\Omega(F)$ and $\delta'>0$ will be chosen later. Let $(B_i)_{i\in I\cup K}$ be a family as in Lemma \ref{lemball}, we obtain
$$\cI_\Omega (F)\leq \frac{1}{1- s^4}\left(\sum_{i\in I} \alpha_{d-1}\rho_i^{d-1}(\nu(n_F(x_i))+\sum_{i\in K} \alpha_{d-1}\rho_i^{d-1}(\nu(n_A(x_i))\right)$$ 
whence, setting $w=s^2/(1+s^2)$,
$$(1-s^2)\cI_\Omega (F)\leq (1- w)\left(\sum_{i\in I} \alpha_{d-1}\rho_i^{d-1}(\nu(n_F(x_i))+\sum_{i\in K} \alpha_{d-1}\rho_i^{d-1}(\nu(n_A(x_i))\right)\,.$$  
Since the balls $(B_i)_{i\in I\cup K}$ are pairwise disjoint, we have
$$V(\cE_n\cap \Omega)\geq \sum_{i\in I \cup K} V(\cE_n\cap B_i)\,.$$
It follows that
\begin{align}\label{eqGint}
&\Prb[V(\cE_n\cap \Omega)\leq (1-s^2)\cI_\Omega(F) n^{d-1},\,\cL^d(\overline{E}_n\Delta F)\leq \ep_{F}\}]\nonumber\\
&\hspace{0.5cm}\leq \Prb\left[ \begin{array}{c}
\sum_{i\in I \cup K} V(\cE_n\cap B_i)\leq (1- w)n^{d-1}\Big(\sum_{i\in I} \alpha_{d-1}\rho_i^{d-1}(\nu(n_F(x_i))\\ +\sum_{i\in K} \alpha_{d-1}\rho_i^{d-1}(\nu(n_A(x_i))\Big) \text{ and }\cL^d(\overline{E}_n\Delta F)\leq \ep_{F}\end{array}\right]\,.
\end{align}
We now choose 
\begin{align}\label{choiceep}
\ep_F\leq\min_{i\in I\cup K}\alpha_d \rho_i^d\delta'
\end{align}
We wish to control $\card ((E_n\cap B_i)\Delta B_i^-)\cap \sZ_n ^d )$, it is equivalent to evaluate $$n^d\cL^d((E_n\cap B_i)\Delta B_i^-)\cap \sZ_n ^d+[-1/2n,1/2n]^d)\,.$$ This was done in \cite{CerfTheret09infc}. We will not redo the computations here, but only state the results:
for $n$ large enough, for $i\in I$,
 $$\card ((E_n\cap B_i)\Delta B_i^-)\cap \sZ_n ^d )\leq 4\delta' \alpha_d \rho_i^d n^d\,.$$
 We recall that $\widetilde{E}_n=E_n\cap \sZ_n^d$. We define $$\widetilde{E}'_n=\widetilde{E}_n\cup (A\cap\sZ_n^d)\text{ and }E'_n= \widetilde{E}'_n+[-1/(2n),1/(2n)]^d\,.$$ For $n$ large enough, for $i\in K$, it was proven in section 5.2 in \cite{CerfTheret09infc} that
  $$\card ((E'_n\cap B_i)\Delta B_i^+)\cap \sZ_n ^d )\leq 4\delta' \alpha_d \rho_i^d n^d\,.$$
Thus, for $n$ large enough, thanks to inequality \eqref{eqGint},
\begin{align}\label{eqG}
&\Prb[V(\cE_n)\leq \lambda n^{d-1}, \,\cL^d (\overline{E}_n\Delta F)\leq \ep_F]\nonumber\\
&\hspace{0.5cm} \leq\sum_{i\in I}\Prb\left[\begin{array}{c} V(\partial_e \widetilde{E}_n\cap B_i)\leq (1-w) \alpha_{d-1}\rho_i^{d-1}(\nu(n_F(x_i))n ^{d-1}\text{,}\\ \card ((\widetilde{E}_n\cap B_i)\Delta (B_i^-\cap \sZ_n^d))\leq 4\delta' \alpha_d \rho_i^d n^d\end{array}\right]\nonumber\\
&\hspace{1cm}+\sum_{i\in K}\Prb\left[\begin{array}{c} V(\partial_e \widetilde{E'}_n\cap B_i)\leq (1-w) \alpha_{d-1}\rho_i^{d-1}(\nu(n_A(x_i))n ^{d-1}\text{,}\\ \card ((\widetilde{E'}_n\cap B_i)\Delta (B_i^-\cap \sZ_n^d))\leq 4\delta' \alpha_d \rho_i^d n^d\end{array}\right]\nonumber\\
&\hspace{0.5cm} \leq\sum_{i\in I\cup K}\Prb[G(x_i,\rho_i,v_i,w,\delta')]\,,
\end{align}
 where $G(x,r,v,w,\delta')$ is the event that there exists a set $U\subset B(x,r)\cap \sZ_n^d$ such that:
 $$\card (U\Delta B^-(x,r,v))\leq  4\delta' \alpha_d r^d n^d$$
 and
 $$V(\partial_e U\cap B(x,r) )\leq (1-w) \alpha_{d-1}r^{d-1}(\nu(n_F(x))n ^{d-1}\,.$$
This event depends only on the edges inside $B(x,r,v)$. This event is a rare event. Indeed, if this event occurs, we can show that the maximal flow from the upper half part of $B(x,r,v)$ (upper half part according to the direction $v$) and the lower half part is abnormally small. To do so, we build from the set $U$ an almost flat cutset in the ball. The fact that $\card (U\Delta B^-(x,r,v))$ is small implies that $\partial_ e U$ is almost flat and is close to $\disc(x,r,v)$. However, this does not prevent the existence of long thin strands that might escape the ball and prevent $U$ from being a cutset in the ball. The idea is to cut these strands by adding edges at a fixed height. We have to choose the appropriate height to ensure that the extra edges we need to add to cut these strands are not too many, so that we can control their capacity. The new set of edges we create by adding to $U$ these edges will be in a sense a cutset. The last thing to do is then to cover  $\disc(x,r,v)$ by hyperrectangles in order to use the estimate that the flow is abnormally small in a cylinder. This work was done in section 6 in \cite{CerfTheret09infc}. It is possible to choose $\delta'$ depending on $F$, $G$ and $w$ such that there exist positive constants $C^F_{1,k}$ and $C^F_{2,k}$ depending on $G$, $d$, $F$, $k$ and $w$ so that for all $k\in I\cup K$,
$$\Prb[G(x_k,\rho_k,v_k,w,\delta')]\leq C^F_{1,k}\exp(-C^F_{2,k}n^{d-1})\,.$$
Using inequality \eqref{eqG}, we obtain 
\begin{align}\label{eqG'}
\Prb[V(\cE_n)\leq \lambda n^{d-1}, \,\cL^d (\overline{E}_n\Delta F)\leq \ep_F]& \leq\sum_{k\in I\cup K} C^F_{1,k}\exp(-C^F_{2,k}n^{d-1})\,.
\end{align}
\noindent Combining inequalities \eqref{eqE}, \eqref{eqF} and \eqref{eqG'}, we obtain that, for small enough $\delta'$,
\begin{align*}
&\Prb[V(\cE_n)\leq \lambda n^{d-1}]\\
&\hspace{0.2cm}\leq C_1\exp(-C_2\beta n^{d-1}) + \sum_{j=1}^N\sum_{i=1}^M \,\,\sum_{k\in I^{F_j}\cup K^{F_j}} C^{F_j}_{1,k}\exp(-C^{F_j}_{2,k}n^{d-1})\\
&\hspace{0.2cm}\leq  C_1\exp(-C_2\beta n^{d-1})+M\sum_{j=1}^N\,\,\sum_{k\in I^{F_j}\cup K^{F_j}} C^{F_j}_{1,k}\exp(-C^{F_j}_{2,k}n^{d-1})\,.
\end{align*}
As $M$, $N$, $|I^{F_j}|$ and $|K^{F_j}|$, for $1\leq j\leq N$, are finite and independent of $n$, we obtain the expected result and this proves Theorem \ref{LLD}.

To conclude, let us sum up the order in which the constants are chosen. We first choose $\delta$ such that it satisfies equality \eqref{choicedelta}. Next, we choose $\delta'$ depending on $\lambda$ and $G$. The parameter $\delta'$ has to satisfy some inequalities that we do not detail here, we refer to section 7 in \cite{CerfTheret09infc}. Finally, to each $F$ in $\sC_\beta$, we choose $\ep_F$ such that it satifies both $\ep_{F_i}\leq \delta$ and inequality \eqref{choiceep}.
\section{Identification of $\varphi_A$}\label{sectionconclusion}
In this section, we prove Proposition \ref{prp}, the last ingredient needed to prove Theorem \ref{thm1}.
\begin{figure}[H]
\def\svgwidth{0.7\textwidth}
\begin{center}
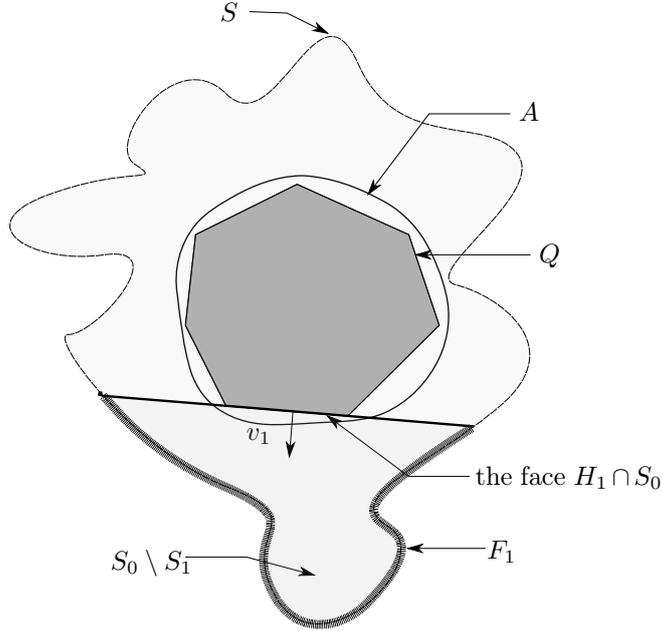
\caption[fig5]{\label{fig5}Construction of $S_1$ for a bounded set $S$ such that $A\subset S$}
\end{center}
\end{figure}
\begin{proof}[Proof of Proposition \ref{prp}]
Let $A$ be a compact convex subset of $\sR^d$. We shall show that any bounded set $S$ that contains $A$ satisfies $\cI(A)\leq \cI(S)$. Let $S$ be such a set, we can assume that $S$ has finite perimeter otherwise the inequality is trivial. Let $\ep>0$. As $A$ is convex, by Lemma \ref{ApproP}, there exists a convex polytope $P$ such that $P\subset A$ and $\cI(A)\leq \cI(P)+\ep$. There exist $m\in\sN^*$, $v_1,\dots,v_m$ unit vectors and $\varphi_1,\dots,\varphi_m$ real numbers such that
$$P=\bigcap_{1\leq i\leq m}\Big\{\,x\in \sR^d  : \, x\cdot v_i\leq \varphi_i\,\Big\}\,.$$
We denote by  $H_i$ the hyperplane associated with $v_i$, \textit{i.e.}, $$H_i= \Big\{\,x\in \sR^d  : \, x\cdot v_i= \varphi_i\,\Big\}\,$$
and $H_i^-$ the associated half-space containing $P$, \textit{i.e},
$$H_i^-= \Big\{\,x\in \sR^d  : \, x\cdot v_i\leq \varphi_i\,\Big\}\,.$$
We shall successively chop off portions from S thereby reducing its surface energy by using the family of half-spaces $(H_i^-,\,1\leq i \leq m)$.
We define by induction this sequence of sets. We set $S_0=S$. Let us assume $S_i$ is already defined for some $i<m$, we set $$S_{i+1}=S_i\cap H_{i+1}^-\,.$$
We next show that $\cI(S_i)\geq \cI(S_{i+1})$ for all $0\leq i< m$. We shall apply the Gauss-Green theorem to each $S_i$ in order to compare the capacity of the face $ H_{i+1}\cap S_i  $ with the capacity of $F_i=\partial S_i\setminus H_{i+1}^-$ (see figure \ref{fig5}). For $i\in\{0,\dots,m-1\}$, let $y_{i+1}\in\cW_\nu$ such that $y_{i+1}$ achieves the supremum in 
$$\sup_{x\in\cW_\nu}\big\{\,x\cdot v_{i+1}\,\big\}\,.$$
 There exists a $\sC ^1$ vector field $f_{i+1}:\sR^d\rightarrow \cW_\nu$ having compact support such that $f_{i+1}(x)=y_{i+1}\in\cW_\nu$ on $\cV(S_i,1)$. We recall that $S_i$ is bounded and we do not go into the details of the existence of such a vector field. Applying Theorem \ref{GGthm} to $S_i\setminus S_{i+1}$  and $f_i$, we obtain
$$\int_{S_i\setminus S_{i+1}} \dive f_{i+1}(x)d\cL^d(x)=\int _{\partial^* (S_i\setminus S_{i+1})}f_{i+1}(x)\cdot n_{S_i\setminus S_{i+1}}(x)d\cH^{d-1}(x)\,.$$
Using Proposition 14.1 in \cite{Cerf:StFlour}, we have $y_{i+1}\cdot v_{i+1}=\nu(v_{i+1})$ and so 
\begin{align*}
\int _{H_{i+1} \cap S_i }y_{i+1}\cdot (-v_{i+1})d\cH^{d-1}(x)=-\nu(v_{i+1})\cH^{d-1}(  H_{i+1}\cap S_i)\,.
\end{align*}
As $f_{i+1}$ is constant on $S_i\setminus S_{i+1}$, we get
\begin{align*}
0=\int_{F_i} f_{i+1}(x)\cdot n_{S}(x)d\cH^{d-1}(x)
-\nu(v_{i+1})\cH^{d-1}(  H_{i+1}\cap S_i)\,,
\end{align*}
and therefore
\begin{align*}
\nu(v_{i+1})\cH^{d-1}(S_i\cap H_{i+1})&=\int_{F_i} f_{i+1}(x)\cdot n_{S}(x)d\cH^{d-1}(x)\\
&\leq  \int_{F_i}  \nu(n_S(x))d\cH^{d-1}(x)\,.
\end{align*}
The last inequality comes from the fact that $f_{i+1}(x)\in\cW_\nu$, therefore we have $$f_i(x)\cdot u\leq \nu(u)$$ for any $u$ in $\sR^d$. 
Finally we obtain as $P\subset A\subset S$ that $S_m=P$ and so $\cI(P)\leq \cI(S)$.
\begin{align*}
\cI(A)\leq \cI(P)+\ep \leq  \cI(S)+\ep\,.
\end{align*}
As this inequality is true for any $\ep>0$, we conclude that $\cI(A)\leq \cI(S)$ and the result follows.
\end{proof}
\noindent Combining Theorem \ref{ULD}, Theorem \ref{LLD} and Proposition \ref{prp}, we obtain Theorem \ref{thm1}.
\paragraph{}
\noindent \textbf{Acknowledgments.} I thank my advisor Marie Th{\'e}ret for presenting me this problem. I also wish to express my gratitude to Rapha{\"e}l Cerf who has given me the opportunity to work with him for an internship. I thank him for our stimulating discussions and for always finding time to talk. This research was partially supported by the ANR project PPPP (ANR-16-CE40-0016).

\bibliographystyle{plain}
\bibliography{biblio}

\end{document}